\documentclass[11pt,reqno,twoside]{article}
\usepackage{amssymb,amsfonts}
\usepackage{amsthm}
\usepackage{amsmath,color,xcolor}
\usepackage{hyperref}
\usepackage[left=1 in,top=1 in,right=1 in,bottom=1 in]{geometry}
\usepackage{cite}
\newcommand{\beq}{\begin{equation}}
\newcommand{\eeq}{\end{equation}}
\newcommand{\beqs}{\begin{equation*}}
\newcommand{\eeqs}{\end{equation*}}
\newcommand{\ba}{\begin{array}}
\newcommand{\ea}{\end{array}}
\newcommand{\beas}{\begin{eqnarray*}}
\newcommand{\eeas}{\end{eqnarray*}}
\newcommand{\bea}{\begin{eqnarray}}
\newcommand{\eea}{\end{eqnarray}}
\newcommand{\bal}{\begin{align}}
\newcommand{\eal}{\end{align}}

\newcommand{\bals}{\begin{align*}}
\newcommand{\eals}{\end{align*}}

\newcommand{\N}{\ensuremath{\mathbb N}}

\newcommand{\norm}[1]{\| {#1} \|}

\newcommand{\bds}{\begin{displaystyle}}
\newcommand{\eds}{\end{displaystyle}}

\renewcommand{\eqref}[1]{(\ref{#1})}

\def\longequals{\mathbin{=\kern-2pt=}}
\def\eqdef{\stackrel{\rm def}{=}}

\def\varep{\varepsilon}

\def\ddt{\frac{d}{dt}}

\newcommand{\remove}[1]{} %-- ON
\renewcommand{\remove}[1]{#1} % OFF

\newtheorem{theorem}{Theorem}[section]

\newtheorem{lemma}[theorem]{Lemma}

\newtheorem{proposition}[theorem]{Proposition}

\newtheorem{remark}[theorem]{\bf{Remark}}

\theoremstyle{remark}
\def\midx{\mu}

\def\myclearpage{}

\definecolor{darkred}{rgb}{.70,.12,.20}

\definecolor{darkgreen}{rgb}{.20,.52,.14}

\numberwithin{equation}{section}

\title{Generalized Forchheimer flows of isentropic gases}
\author{Emine Celik$^a$, Luan Hoang$^a$ and Thinh Kieu$^b$}

\date{\today}

\begin{document}
\maketitle 
\begin{center}
\textit{$^a$Department of Mathematics and Statistics, Texas Tech University, Box 41042, Lubbock, TX 79409--1042, U. S. A.} \\
\textit{$^b$Department of Mathematics, University of North Georgia, Gainesville Campus, 3820 Mundy Mill Rd., Oakwood, GA 30566, U. S. A.}\\
Email addresses: \texttt{emine.celik@ttu.edu, luan.hoang@ttu.edu, thinh.kieu@ung.edu}\\
%$^*$Corresponding author
\end{center}
%\setcounter{equation}{0}
%-----------------------------------------------------------------------------------------------------------------
\begin{abstract}
We consider generalized Forchheimer flows of either isentropic gases or slightly compressible fluids in porous media. 
By using Muskat's and Ward's general form of the Forchheimer equations, we describe the fluid dynamics by  a doubly nonlinear parabolic equation for the appropriately defined pseudo-pressure. The volumetric flux boundary condition is converted  to a time-dependent Robin-type boundary condition for this pseudo-pressure.
We study the corresponding initial boundary value problem, and  estimate the $L^\infty$ and $W^{1,2-a}$ (with $0<a<1$) norms for the solution  on the entire domain  in terms of the initial and boundary data. It is carried out by using a suitable  trace theorem and an appropriate modification of Moser's iteration.
%\textcolor{violet}{In this paper, we consider the system of equations that describes the flow for a single-phase slightly compressible fluid in porous media subject to generalized Forchheimer's law. Using dimension argument, the original system is reduced to  a ``pseudo-pressure'' equation of doubly nonlinear degenerate parabolic type. The global estimates in $L^\infty$ and $W^{1,2-a}$-norms for the solution of above equation with Robin-type boundary condition are established by using a new trace theorem, and an appropriate modification of Moser's iteration.}
\end{abstract}
%-----------------------------------------------------------------------------------------------------------------

%\tableofcontents 
%   \newpage \pagenumbering{arabic}
   %\tableofcontents 

\pagestyle{myheadings}\markboth{E. Celik, L. Hoang, and T. Kieu}
{Generalized Forchheimer flows of isentropic gases}

%%%%%%%%%%%%%%%%%%%%%%%%%%%%%%%%%%%%%%%%%%%%%%%%%%%%%%%%%%%%%%%%%%%%%
\myclearpage    
\section{Introduction}
%%%%%%%%%%%%%%%%%%%%%%%%%%%%%%%%%%%%%%%%%%%%%%%%%%%%%%%%%%%%%%%%%%%%%

The most common equation to describe fluid flows in porous media is the Darcy law  
\beq\label{Darcy}
-\nabla p = \frac {\mu}{k} v,
\eeq
where $p$, $v$, $\mu$, $k$ are, respectively (resp.), the pressure, velocity, absolute viscosity and permeability.

However, this linear equation  is not valid in many situations, particularly, when the Reynolds number increases, see \cite{Muskatbook,BearBook}. 
%Even in early works,  Darcy \cite{Darcybook} and Dupuit \cite{Dupuit1857} already acknowledged the deviations from equation \eqref{Darcy}.
Even in the early work,  Darcy \cite{Darcybook} already acknowledged the deviations from equation \eqref{Darcy}.
There have been many investigations into what equations for hydrodynamics in porous media to replace Darcy's law \eqref{Darcy}, see \cite{Muskatbook,Ward64,BearBook,NieldBook,StraughanBook} and references therein. Forchheimer \cite{Forchh1901,ForchheimerBook} established the following three nonlinear empirical models:  
two-term Forchheimer equation
\beq\label{2term}
-\nabla p=av+b|v|v,
\eeq
three-term Forchheimer equation
\beq\label{3term}
-\nabla p=av+b|v|v+c |v|^2 v,
\eeq
and Forchheimer's power law
\beq\label{power}
-\nabla p=av+d|v|^{m-1}v,\quad\text{for some real number } m\in(1, 2).
\eeq
Above, the positive constants $a,b,c,d$ are obtained from experiments.

While mathematics of Darcy's flows have been studied intensively for a long time with vast literature, see e.g. \cite{VazquezPorousBook},
there is a much smaller number of mathematical papers on Forchheimer flows and they appeared much later. Among those, there are even fewer papers dedicated to compressible fluids.
(See \cite{StraughanBook} and references there in.)

% In order to cover more general nonlinear flows in porous media formulated from experiments that differ from Darcy's law and three types of Forchheimer's equations,
% generalized Forchheimer equations were proposed.% \cite{ABHI1,Doug1993}.
In order to cover general nonlinear flows in porous media formulated from experiments,
generalized Forchheimer equations were proposed. They extend the models \eqref{2term}--\eqref{power} and are of the form 
\beq\label{gF}
-\nabla p =\sum_{i=0}^N a_i |v|^{\alpha_i}v. 
\eeq  
%where $a_i$ are positive constants.  
These equations are analyzed numerically in \cite{Doug1993,Park2005, Kieu1},
theoretically in \cite{ABHI1,HI2,HIKS1,HKP1,HK1,HK2} for single-phase flows, and also in \cite{HIK1,HIK2} for two-phase flows.

Our previous analysis  \cite{ABHI1,HI2,HIKS1,HKP1,HK1,HK2}  was focused on a simplified model for slightly compressible fluids.
Though such a minor simplification is commonly used in reservoir engineering, the mathematical rigor is compromised.
Furthermore, since the model does not specify the dependence on the density, its applications to gaseous flows 
would be inaccurate and might present artificial   technical difficulties.
The goals of this paper are: 
(a) Developing a more accurate model for generalized Forccheimer equations for gases, and
(b) Analyzing it without making any simplifications.

For goal (a), we first have to modify \eqref{gF} to reflect the dependence on the density.
We return to an idea by Muskat and  Ward.
By using dimension analysis, Muskat \cite{Muskatbook} and then Ward \cite{Ward64} proposed the following equation for both laminar and turbulent flows in porous media:
\beq\label{W}
-\nabla p =f(v^\alpha k^{\frac {\alpha-3} 2} \rho^{\alpha-1} \mu^{2-\alpha}),\text{ where  $f$ is a function of one variable.}
\eeq 

In particular, when $\alpha=1,2$, Ward \cite{Ward64} established from experimental data that
\beq\label{FW} 
-\nabla p=\frac{\mu}{k} v+c_F\frac{\rho}{\sqrt k}|v|v,\quad \text{where }c_F>0.
\eeq

Combining  \eqref{gF} with the suggestive form \eqref{W} for the dependence on $\rho$ and $v$, we propose the following equation 
 \beq\label{FM}
-\nabla p= \sum_{i=0}^N a_i \rho^{\alpha_i} |v|^{\alpha_i} v,
 \eeq
where $N\ge 1$, $\alpha_0=0<\alpha_1<\ldots<\alpha_N$ are real numbers, the coefficients $a_0, \ldots, a_N$ are positive.  

Here, the viscosity and permeability are considered constant and we do not specify the dependence of $a_i$'s on them.
Our mathematical exposition below will allow all $\alpha_i\ge 1$  in \eqref{FM}. 
In practice, we can simply take $\alpha_N\le 2$ in \eqref{FM} or use the popular model \eqref{FW}.
Even in these cases,  the results obtained in this paper are still new.

%Note that the dependence on viscosity in \eqref{W} raises the concern on its physical validity when $\alpha>2$, see \cite[sec. 2.2, Chap. II]{Muskatbook}. Nonetheless, the empirical model  \eqref{2term} has $\alpha=3$. Moreover, our focus is on the relation between pressure, density and velocity, but not the viscosity. 
%Therefore, the following mathematical exposition will allow all $\alpha_i\ge 1$  in \eqref{FM}. Even for the case $\alpha_N\le 2$ in \eqref{FM} or the popular model \eqref{3term}, the results obtained in this paper are still new.

Multiplying both sides of  \eqref{FM}  by $\rho$ gives
 \beq\label{eq1}
 g( |\rho v|) \rho v   =-\rho\nabla p,
 \eeq
where $g:\mathbb{R}^+\rightarrow\mathbb{R}^+$ is a generalized polynomial with positive coefficients defined by
\beq\label{eq2}
g(s)=a_0s^{\alpha_0} + a_1s^{\alpha_1}+\cdots +a_Ns^{\alpha_N}\quad\text{for } s\ge 0.
\eeq 
 
We will study the following two types of compressible fluids.

\emph{1. Isentropic gases.} For isentropic gases, the constitutive law is
\beq\label{gas}
p=c\rho^\gamma\quad\text{for some } c,\gamma>0.
\eeq
Then from \eqref{eq1} and \eqref{gas} follows
\beq\label{ru1}
 g( |\rho v|) \rho v   =-\rho\nabla p=-\nabla u\quad \text{with } u=\frac{c\gamma\rho^{\gamma+1}}{\gamma+1}.
\eeq
Solving for $\rho v$ from this equation yields
 \beq\label{ru} 
\rho v=-K(|\nabla u|)\nabla u,
\eeq
where the function $K: \mathbb{R}^+\rightarrow\mathbb{R}^+$ is defined for $\xi\ge 0$ by
\beq \label{Kxidef}
K(\xi)=\frac{1}{g(s(\xi))}  \text{ with  }s=s(\xi) \text{ being the unique non-negative solution of }
sg(s)=\xi.
\eeq

Recall the continuity equation
\beq\label{con-law}
\phi\rho_t+{\rm div }(\rho v)=0,\eeq
where constant $\phi\in(0,1)$ is the porosity.
Rewrite 
\beq\label{maineq1}
\rho=(\frac{\gamma+1}{c\gamma})^\frac1{\gamma+1} u^\lambda \quad\text{with } \lambda=\frac1{\gamma+1}\in (0,1).
\eeq
Combining \eqref{con-law} and \eqref{ru} with relation \eqref{maineq1}, we have
\beq\label{maineq2}
(u^\lambda)_t= \frac{1}{\phi} \Big(\frac{c\gamma}{\gamma+1}\Big)^\lambda\, \nabla\cdot(K(|\nabla u|)\nabla u)).
\eeq

\emph{2. Slightly compressible fluids.} The equation of state is
\beq\label{slight}
\frac 1\rho \frac{d\rho}{dp}=\frac 1\kappa=const.>0.
\eeq
Note from \eqref{slight} that $\rho \nabla p=\kappa\nabla \rho$.
Then combining this with \eqref{eq1} gives
\beq\label{ru0}
 g( |\rho v|) \rho v   =-\nabla u\quad \text{with } u=\kappa \rho.
\eeq
which is the same as \eqref{ru1}. Thus, we obtain formula \eqref{ru} for $\rho v$. 
By combining \eqref{ru} and \eqref{con-law} we have
\beq\label{utporo}
u_t=\frac\kappa\phi\nabla \cdot(K(|\nabla u|)\nabla u).
\eeq

Observe that we can write $u_t=(u^\lambda)_t$ with $\lambda=1$ in \eqref{utporo}. Therefore, the two equations \eqref{maineq2} and \eqref{utporo}  are the same except for the factors on the right-hand sides. Then by scaling the time variable in both \eqref{maineq2} and \eqref{utporo}, we obtain, for both isentropic gases and slightly compressible fluids, the unified equation
\beq\label{maineq}
(u^\lambda)_t= \nabla\cdot(K(|\nabla u|)\nabla u))\quad\text{with } \lambda\in(0,1].
\eeq
This will be the partial differential equation (PDE) of our interest. 
It is derived and will be analyzed without any simplifications (goal (b) above).
Although the case of isentropic gases, i.e. $\lambda<1$, is the main focus, the analysis will also cover the case of slightly compressible fluids, i.e. $\lambda=1$, at no extra cost.

In case of ideal gases, i.e., $\gamma =1$, we can derive from \eqref{maineq} a PDE for pseudo-pressure $p^2$. In general, we rewrite $u$ in \eqref{ru1} as $u=c' p^\frac{\gamma+1}{\gamma}$ for some $c'>0$, 
hence it is approximately the pseudo-pressure for isentropic gases \eqref{gas}.
For simplicity, we refer to $u$ as the pseudo-pressure. Therefore,  equation \eqref{maineq} is a PDE describing the dynamics of the pseudo-pressure $u$. 

Regarding the boundary conditions, we focus on the volumetric flux condition $v\cdot \vec\nu=\psi$, which results  in $\rho v\cdot \vec\nu=\psi \rho$, or,
\beq\label{BC}
-K(|\nabla u|)\nabla u\cdot \vec\nu=\varphi u^\lambda,
\eeq
where $\lambda\in(0,1)$, $\varphi=(\frac{\gamma+1}{c\gamma})^\frac1{\gamma+1}\psi$ in case isentropic gases, and $\lambda =1$, $\varphi=\psi/\kappa$ in case of slightly compressible fluids.  Here, $\vec \nu$ is the outward normal vector on the boundary.

From mathematical point of view, equation \eqref{maineq} for $\lambda<1$ is a doubly nonlinear parabolic equation, which is an interesting topic of its own.
Research on doubly nonlinear parabolic equations follows the development of 
general parabolic equations \cite{LadyParaBook68,LiebermanPara96} and degenerate/singular parabolic equations \cite{DiDegenerateBook,DiHarnackBook}.
However, it requires much more complicated techniques.
See monograph \cite{IvanovBook82}, review paper \cite{IvanovReg1997} and references therein.
%, and a series of his papers which, not exclusively, include 
%\cite{ivanov1999gradient,IvanovUniform1991,IvanovReg1997}. 
%\cite{ivanov1991holder,ivanov1997holder,ivanov2000existence,ivanov1999gradient,MR1137526}. 
For other developments, see e.g.  \cite{kinnunen2007local,Tsutsumi1988,ManfrediVespri,Vespri1992,Alt1983}.

There are two issues that did not attract attention in most existing papers:
(1) Robin boundary condition, and
(2) Estimates for super-critical case. 
For instance, \cite{Tsutsumi1988,ManfrediVespri} give global $L^\infty$-estimates but for homogeneous Dirichlet boundary condition in the sub-critical case, see discussion in Remark \ref{asmall} below.
Regarding interior estimates, the common result is
\beq\label{comineq}
\|u\|_{L^\infty(Q(R/2,T/2))}\le C_{R,T} (\|u\|_{L^q(Q(R,T))}^{\alpha(q)}+1),
\eeq
where $q\in(1,\infty)$, $C_{R,T}>0$ depends explicitly on $R$ and $T$,  while $Q(R,T)$ denotes the cylinder $B_R\times(-T,0)$.
Surnachev \cite{Surnachev2012} improves it to
\beq\label{Mishaineq}
\|u\|_{L^\infty(Q(R/2,T/2))}\le C_{R,T} (\|u\|_{L^q(Q(R,T))}^{\alpha(q)}+\|u\|_{L^q(Q(R,T))}^{\beta(q)})\quad\text{with }\alpha(q),\beta(q)>0.
\eeq
We call \eqref{Mishaineq} a \emph{quasi-homogeneous estimate} (with respect to $\|u\|_{L^q(Q(R,T))}$).
The global version of \eqref{Mishaineq} is not known for doubly nonlinear equations, though it was established for degenerate equations, see e.g. \cite{HK2}.

In this paper, we focus on both topics (1) and (2) listed above.
In this case, the boundary condition \eqref{BC} gives rise to high power boundary integral which cannot be treated by the standard trace theorem.
Therefore, we derive and utilize a new, suitable trace inequality to obtain bounds for the solutions of \eqref{maineq} in terms of initial and boundary data.  
For $L^\infty$-estimates, we make some technical improvements in order to overcome the non-homogeneity of function $K(\cdot)$ and non-zero boundary data. We carefully modify Moser's iteration \cite{moser1971pointwise} and obtain quasi-homogeneous estimates. Our results are for both (spatially) interior and global estimates, hence, extend the previous interior improvement \eqref{Mishaineq}.
% to the global version for the IBVP \eqref{rho:eq}.

% We obtain bounds for the solutions in terms of initial and boundary data by using suitable  trace theorems and multiplicative Sobolev inequalities.
% We also make some technical improvements with these estimates. To overcome the non-homogeneity of function $K(\cdot)$, we carefully modify Moser's iteration \cite{moser1971pointwise} to obtain quasi-homogeneous estimates. Our results are for both (spatially) interior and global estimates, hence, extend the improvement in \cite{Surnachev2012}.
% to the global version for the IBVP \eqref{rho:eq}.

Throughout this paper,  $U$ is an open, bounded subset of $\mathbb{R}^n$, with $n=2,3,\ldots$, and has $C^1$-boundary $\Gamma=\partial U$.
For physics problems $n=2,3$, but we consider here any natural number $n\ge 2$. 
Hereafter, we fix the functions $g(s)$ in \eqref{eq1} and \eqref{eq2}. 
Therefore, the exponents $\alpha_i$ and coefficients $a_i$ are all fixed, and so is the function $K(\xi)$  in \eqref{Kxidef}.  
Also, our calculations frequently use the following exponent
\beq\label{eq9}
a=\frac{\alpha_N}{\alpha_N+1}\in (0,1).
\eeq

We consider the initial boundary value problem associated with \eqref{maineq} and \eqref{BC}, specifically,
\beq\label{rho:eq}
\begin{cases}
\frac{\partial (u^{\lambda})}{\partial t} = \nabla \cdot (K (|\nabla u|)\nabla u  ) &\text {in }  U\times (0,\infty),\\
u(x,0)=u_0(x) &\text {in } U,\\
K(|\nabla u|)\frac{\partial u}{\partial \vec\nu}+\varphi u^\lambda=0  &\text{on } \Gamma \times(0,\infty),
\end{cases}
\eeq 
where $u_0(x)$ and $\varphi(x,t)$ are given initial and boundary data, respectively. Again, $\vec\nu$ denotes the outward normal vector on $\Gamma$.
Here, $\lambda$ is a fixed number in $(0,1]$ for the remaining of the paper. 

The current article is focused on studying non-negative solutions of problem \eqref{rho:eq}.
Section \ref{Prelim} contains new trace theorems and multiplicative Sobolev's inequalities, which are suitable to the Robin-type boundary condition \eqref{BC}, as well as the nature of our equation's double nonlinearity.
%%%%
In section \ref{Lalpha}, we estimate $L^\alpha$-norms of the solutions for all $\alpha>0$, in terms of initial and boundary data. 
These will also be used for later gradient and $L^\infty$ estimates.     
%%%
In section \ref{GradSec}, we present estimates for the gradient's $L^{2-a}$-norm for time $t>0$. 
%%%
In section \ref{Linterior}, we estimate the $L^\infty$-norm of the solution in any compact subsets of the domain.
%Cacciappolli inequality of nonhomogeneous type is established in Lemma \ref{interiorCac}.
Due to the basic Lebesgue norm relation  in Proposition \ref{preMoser}, the Moser's iteration is of a non-homogeneous form \eqref{Moserform1}.
We deal with this by using  Lemma \ref{Genn} and obtain in Theorem \ref{Linf1} the quasi-homogeneous estimate.
% The estimates in terms of initial and boundary data are obtained in Theorem \ref{Linfcompact}.
%%%
Section \ref{Lglobal} is focused on estimating the solution's $L^\infty$-norm on the entire domain.
In Proposition \ref{GLk}, the $L^{\kappa\alpha}$-norm is bounded by the $L^{\alpha+\mu_{1}}$-norm, where, unlike the interior case, both $\kappa>1$ and $\mu_1>0$ depend on $\alpha$,
and the constant  depends on the boundary data.
To manage the powers during iterations, we construct in Lemma \ref{al-be-sq} two controlling sequences $(\alpha_j)_{j=0}^\infty$ and $(\beta_j)_{j=0}^\infty$.
Using these sequences for iterations, we obtain in Theorem \ref{LinfU} the quasi-homogeneous estimates for the $L^\infty$-norm, and in
Theorem \ref{LinfData} the ultimate estimates in terms of initial and boundary data.
%%%
The Appendix contains key Lemma \ref{Genn} in implementing Moser's iteration with non-homogeneous inequalities.%
\myclearpage
\section{Auxiliaries}\label{Prelim}
%%%%%%%%%%%%%%%%%%%%%%%%%%%%%%%%%%%%%%%%%%%%%%%%%%%%%%%%%%%%%%%%%%%%%

First, we recall elementary inequalities that will be used frequently. 
Let $x,y\ge 0$. Then
\beq\label{ee1}
(x+y)^p\le 2^p(x^p+y^p)\quad  \text{for all }  p>0,
\eeq
\beq\label{ee2}
(x+y)^p\le x^p+y^p\quad  \text{for all } 0<p\le 1,
\eeq
\beq\label{ee3}
(x+y)^p\le 2^{p-1}(x^p+y^p)\quad  \text{for all }  p\ge 1,
\eeq
\beq\label{ee4}
x^\beta \le x^\alpha+x^\gamma\quad \text{for all } 0\le \alpha\le \beta\le\gamma,
\eeq
particularly,
\beq\label{ee5}
x^\beta \le 1+x^\gamma \quad \text{for all } 0\le \beta\le\gamma.
\eeq

Second, we establish particular Poinc\'are-Sobolev inequality and trace theorem for studying our doubly nonlinear equation with Robin-type boundary condition. 

For any $1\le p<n$, we denote by $p^*$ its Sobolev conjugate exponent, that is,
$
p^*=\frac{np}{n-p}.
$

%Let $n\ge 2$ and $U$ be  an open, bounded subset in $\mathbb{R}^n$ with $C^1$ boundary $\Gamma=\partial U$.
%==================================================================%
\begin{lemma}\label{gentrace}
In the following statements, $u(x)$ is a function  defined on $U$.
%, with a well-defined trace on $\Gamma$.
\begin{enumerate}
\item If $\alpha \ge s\ge 0$, $\alpha\ge 1$, and $p>1$, then for any $|u|^\alpha\in W^{1,1}(U)$ and $\varepsilon>0$ one has
\beq\label{trace10}
\int_\Gamma |u|^\alpha d\sigma 
\le \varepsilon \int_U |u|^{\alpha-s}|\nabla u|^p dx + c_1 \int_U |u|^\alpha dx  + (c_2 \alpha)^\frac p{p-1} \varepsilon^{-\frac 1{p-1}} \int_U |u|^{\alpha+\frac{s-p}{p-1}} dx,
\eeq
where   $c_1,c_2>0$ are constants depending on $U$, but not on $u(x),\alpha,s,p$.

\item If  $n>p>1$, $r>0$, $\alpha\ge s\ge 0$,  $\alpha\ge \frac{p-s}{p-1}$, and $\alpha>\frac{n(r+s-p)}{p}$, then for any $\varepsilon>0$ one has
\begin{multline}\label{S10}
\int_U |u|^{\alpha+r}dx
\le \varepsilon \int_U |u|^{\alpha-s}|\nabla u|^p dx+ \varepsilon^{-\frac \theta{1-\theta}} 2^\frac{\theta(\alpha-s+p)}{1-\theta}(c_3 m)^{\frac{\theta p}{1-\theta}} \|u\|_{L^\alpha}^{\alpha+\midx_1}\\  
 +2^{\theta (\alpha-s+p)}c_4^{\theta p}|U|^\frac{\theta(\alpha(p-1)+s-p)}{\alpha}\|u\|_{L^\alpha}^{\alpha+r},
\end{multline}
for all $|u|^m\in W^{1,p}(U)$, where  
\beq\label{mdef}
 m=\frac{\alpha-s+p}p,\quad
\theta=\frac{rn}{n(p-s)+\alpha p},\quad
\midx_1=\frac{r+\theta(s-p)}{1-\theta},
\eeq
and constants $c_3,c_4>0$ depend on $U,p$, but not on $u(x),\alpha$, $s$.

\item If $n>p>1$, $\alpha\ge s>p$,  and $\alpha> \frac{n(s-p)}{p-1}$,  
then for any $\varepsilon>0$, one has
\begin{multline}\label{trace30}
\int_\Gamma |u|^\alpha d\sigma 
\le 2\varepsilon \int_U |u|^{\alpha-s}|\nabla u|^p dx + c_1 \|u\|_{L^\alpha}^\alpha\\
+ \varepsilon^{-\frac 1{p-1}}  D_{1,\alpha}\|u\|_{L^\alpha}^{\alpha+r}
\quad  + \varepsilon^{-(\frac 1{p-1}+\frac p{p-1}\frac \theta{1-\theta})}  D_{2,\alpha} \|u\|_{L^\alpha}^{\alpha+\midx_1},
\end{multline}
for all functions $u(x)$ satisfying $|u|^\alpha\in W^{1,1}(U)$ and $|u|^m\in W^{1,p}(U)$,
where 
 $m$ is defined by \eqref{mdef}, 
\beq\label{thetaorig} 
\theta=\frac{1}{(p-1)(\frac{\alpha p}{n(s-p)}-1)},
\eeq
\beq\label{d12}
D_{1,\alpha}= 2^{\theta(\alpha-s+p)}(c_2 \alpha)^\frac p{p-1}  c_4^{\theta p}|U|^\frac{\theta(\alpha(p-1)+s-p)}{\alpha},\quad
D_{2,\alpha}=  2^\frac{\theta(\alpha-s+p)}{1-\theta} (c_2\alpha)^{\frac p{p-1}\frac 1{1-\theta}} (c_3 m)^\frac{\theta p}{1-\theta}.
\eeq
\end{enumerate}
\end{lemma}
%==================================================================%
\begin{proof}
We make a couple of comments before starting the proof.
First, the boundary integrals in \eqref{trace10} and \eqref{trace30} are in the sense of traces of $|u|^\alpha$ on $\Gamma$.
Second, observe that the conditions on $u(x)$ do not guarantee that the right-hand sides of  \eqref{trace10}, \eqref{S10} and \eqref{trace30} are finite. In case they are not, these inequalities are understood to be trivially true.
Therefore, the following proof only needs to cover the case when those right-hand sides are finite.

(i) We recall the trace theorem
\beqs
\int_\Gamma |\phi|d\sigma \le c_1\int_U |\phi|dx+c_2\int_U |\nabla \phi|dx,
\eeqs
for all $\phi \in W^{1,1}(U)$, where $c_1$ and $c_2$ are positive constants depending on $U$.
Applying this trace theorem to $\phi=|u|^{\alpha}$, we have 
\beq\label{trualpha}
\begin{aligned}
\int_\Gamma |u|^\alpha d\sigma \le c_1 \int_U |u|^\alpha dx + c_2\alpha \int_U |u|^{\alpha-1}|\nabla u|dx.
\end{aligned}
\eeq
Rewriting $c_2\alpha |u|^{\alpha-1}|\nabla u|$ in the last integral as a product of  $\varepsilon^{1/p}u^{\frac{\alpha-s}p}|\nabla u|$ and $c_2\alpha \varepsilon^{-1/p} u^\frac{(p-1)\alpha+s-p}{p}$, and applying Young's inequality with exponent $p$ and $p/(p-1)$,  we obtain  inequality \eqref{trace10}.

(ii)  Since $\alpha\ge s$, the number $m$ defined by \eqref{mdef} is greater or equal to $1$.
Then applying Sobolev-Poincar\'e  inequality to $|u|^m$ yields
\beq\label{sobwm}
\begin{aligned}
 \|\,|u|^m\,\|_{L^{p^*}} \le c_3 \|\nabla (|u|^m)\|_{L^p} +c_4\int_U |u|^m dx,
\end{aligned}
\eeq
where $c_3$ and $c_4$ are positive constants depending on $U$ and $p$. 
Note that by definition \eqref{mdef} of $m$, we have $(m-1)p=\alpha-s$. Hence \eqref{sobwm} can be written as
\beqs
\Big(\int_U |u|^{p^*m} dx\Big)^{1/p^*}
\le c_3 m\Big(\int_U |u|^{\alpha-s}|\nabla u|^p dx\Big)^{1/p} +c_4 \int_U |u|^m dx.
\eeqs
Raising both sides to the power $1/m\le 1$ and using inequality \eqref{ee2},  we obtain
\beq\label{ti1}
\Big(\int_U u^{p^*m} dx\Big)^\frac{1}{p^*m}
\le (c_3 m)^\frac1m \Big(\int_U u^{\alpha-s}|\nabla u|^p dx\Big)^\frac1{\alpha-s+p} +c_4^\frac1m\Big( \int_U u^m dx\Big)^\frac1m .
\eeq
Note that
\beq\label{qdef}
q=p^*m=\frac{n(\alpha-s+p)}{n-p}.
\eeq
Then \eqref{ti1} yields
\beq\label{Sov10}
\|u\|_{L^q}\le (c_3 m)^\frac 1m \Big(\int_U u^{\alpha-s}|\nabla u|^p dx\Big)^\frac1{\alpha-s+p} +c_4^\frac1m\|u\|_{L^m}.
\eeq
Since $n>p$, $r>0$, and $\alpha>\frac{n(r+s-p)}{p}$, one has
$\alpha< \alpha+r <q$.
Then 
\beqs
\frac1{\alpha+r}
=\frac{\theta_0}{q}+\frac{1-\theta_0}{\alpha},
\eeqs
where $\theta_0\in (0,1)$  is defined by
\beq\label{thezero}
\theta_0= \frac{rq}{(\alpha+r)(q-\alpha)} \quad \text{with }q=p^*m,\eeq 
Then interpolation inequality and \eqref{Sov10} give
\begin{align*}
 \| u\|_{L^{\alpha+r}}
&\le \| u\|_{L^q}^{\theta_0}\|u\|_{L^\alpha}^{1-\theta_0}
\le \Big\{ (c_3 m)^\frac1m \Big(\int_U u^{\alpha-s}|\nabla u|^p dx\Big)^\frac1{\alpha-s+p} +c_4^\frac1m\|u\|_{L^m} \Big\}^{\theta_0}\|u\|_{L^\alpha}^{1-\theta_0}.
\end{align*}
Raising both sides to the power $\alpha+r$  and applying inequality \eqref{ee1} for exponent $\theta_0 (\alpha+r)$ yield
\beq\label{ualpr}
\begin{aligned}
\int_U u^{\alpha+r} dx
&\le 2^{\theta_0 (\alpha+r)}\Big\{ (c_3 m)^\frac{\theta_0(\alpha+r)}{m} \Big(\int_U u^{\alpha-s}|\nabla u|^p dx\Big)^\frac{\theta_0(\alpha+r)}{\alpha-s+p} +c_4^\frac{\theta_0(\alpha+r)}{m}\|u\|_{L^m}^{\theta_0(\alpha+r)} \Big\}\|u\|_{L^\alpha}^{(1-\theta_0)(\alpha+r)}\\
&= 2^{\theta_0 (\alpha+r)}(c_3 m)^\frac{\theta_0 (\alpha+r)}{m} \|u\|_{L^\alpha}^{(1-\theta_0)(\alpha+r)} \Big(\int_U u^{\alpha-s}|\nabla u|^p dx\Big)^\theta\\
&\quad +2^{\theta_0 (\alpha+r)}c_4^\frac{\theta_0 (\alpha+r)}{m}\|u\|_{L^\alpha}^{(1-\theta_0)(\alpha+r)}\|u\|_{L^m}^{\theta_0(\alpha+r)},
\end{aligned}
\eeq
where 
\beq\label{thetatemp}
\theta= \frac{\theta_0(\alpha+r)}{\alpha-s+p}\text{ which is the same as in \eqref{mdef}.}
\eeq
Since $\alpha >\frac{n(r+s-p)}p$, we have  $\theta\in(0,1)$. Then applying Young's inequality to the first term on the right-hand side of \eqref{ualpr} with powers $\frac 1\theta$ and $\frac 1{1-\theta}$, we obtain
\begin{multline}\label{S2}
\int_U u^{\alpha+r}dx
\le \varepsilon \int_U u^{\alpha-s}|\nabla u|^p dx+ \varepsilon^{-\frac \theta{1-\theta}} 2^\frac{\theta_0(\alpha+r)}{1-\theta}(c_3 m)^{\frac{\theta_0 (\alpha+r)}{m}\frac 1{1-\theta}} \|u\|_{L^\alpha}^{(1-\theta_0)(\alpha+r)\frac 1{1-\theta}}\\ 
\quad +2^{\theta_0 (\alpha+r)}c_4^\frac{\theta_0 (\alpha+r)}{m}\|u\|_{L^\alpha}^{(1-\theta_0)(\alpha+r)}\|u\|_{L^{m}}^{\theta_0(\alpha+r)}.
\end{multline}

Since $\alpha\ge \frac{p-s}{p-1}$, then $m\le \alpha$. By applying H\"older's inequality to bound the $L^m$-norm of $u$ on the right-hand side of \eqref{S2} by $|U|^{\frac 1m- \frac 1\alpha}\|u\|_{L^\alpha}$, we obtain
\begin{multline}\label{S11}
\int_U |u|^{\alpha+r}dx
\le \varepsilon \int_U |u|^{\alpha-s}|\nabla u|^p dx+ \varepsilon^{-\frac \theta{1-\theta}} 2^\frac{\theta_0(\alpha+r)}{1-\theta}(c_3 m)^{\frac{\theta_0 (\alpha+r)}{m}\frac 1{1-\theta}} \|u\|_{L^\alpha}^{(1-\theta_0)(\alpha+r)\frac 1{1-\theta}}\\  
 +2^{\theta_0 (\alpha+r)}c_4^\frac{\theta_0 (\alpha+r)}{m}|U|^{\theta_0 (\alpha+r)(\frac 1m- \frac 1\alpha)}\|u\|_{L^\alpha}^{\alpha+r}.
\end{multline}
Re-calculations of the powers:
$$\theta_0(\alpha+r)=\theta(\alpha-s+p), \quad \theta_0(\alpha+r)/m=\theta p,
$$
\begin{align*}
(1-\theta_0)(\alpha+r)\frac 1{1-\theta}
&=\Big(1-\theta\frac{\alpha-s+p}{\alpha+r}\Big)\frac{\alpha+r}{1-\theta}=\frac{(1-\theta)\alpha+r+\theta(s-p)}{1-\theta}\\
&=\alpha+\frac{r+\theta(s-p)}{1-\theta}=\alpha+\midx_1,
\end{align*}
\begin{align*}
\theta_0 (\alpha+r)(\frac 1m- \frac 1\alpha)
&=\theta(\alpha-s+p)\Big(\frac{p}{\alpha-s+p}-\frac1\alpha\Big)
=\frac{\theta(\alpha(p-1)+s-p)}{\alpha}.
\end{align*}
Thus, inequality \eqref{S10} follows \eqref{S11}.

(iii) Define 
\beq\label{rdef} 
r=\frac{s-p}{p-1}.
\eeq

Given $\varepsilon>0$. First, we apply inequality \eqref{trace10}, and then estimate the last integral $\int_U |u|^{\alpha+r}dx$
in \eqref{trace10} by  using \eqref{S10} with the parameter $\varepsilon$ in  \eqref{S10} being set as
\beqs
\varepsilon (c_2\alpha)^{-\frac p{p-1}} \varepsilon^{\frac 1{p-1}} =\varepsilon^{\frac p{p-1}}(c_2 \alpha)^{-\frac p{p-1}}
=(\varepsilon^{-1}c_2\alpha)^{-\frac p{p-1}}. 
\eeqs 
This results in \begin{multline*}
\int_\Gamma u^\alpha d\sigma 
\le 2\varepsilon \int_U u^{\alpha-s}|\nabla u|^p dx + c_1 \int_U u^\alpha dx \\
\quad  + (c_2 \alpha)^\frac p{p-1} \varepsilon^{-\frac 1{p-1}}   (\varepsilon^{-1}c_2\alpha)^{\frac p{p-1}\frac \theta{1-\theta}} 2^\frac{\theta(\alpha+s-p)}{1-\theta}(c_3 m)^{\frac{\theta p}{1-\theta}} \|u\|_{L^\alpha}^{\alpha+\midx_1}\\ 
\quad + \varepsilon^{-\frac 1{p-1}}  (c_2 \alpha)^\frac p{p-1}  2^{\theta (\alpha-s+p)}c_4^{\theta p}|U|^\frac{\theta(\alpha(p-1)+s-p)}{\alpha}\|u\|_{L^\alpha}^{\alpha+r},
\end{multline*}
where $\theta$ is defined in \eqref{mdef}, which is the same as in \eqref{thetaorig} since $r$ is now specified by \eqref{rdef}.
Therefore \eqref{trace30} follows. Finally, we check the conditions on the exponents in (i) and (ii) to validate our calculations. 
Since $\alpha\ge s>p>1$, we have $r>0$, $(p-s)/(p-1)<0$, and only need to check  $\alpha>\frac{n(r+s-p)}{p}$. With $r$ defined by \eqref{rdef}, this, in fact, is $\alpha>n(s-p)/(p-1)$, which is already one of the assumptions on $\alpha$. 
The proof is complete.
\end{proof}

%==================================================================%

In our particular case, we have following lemma. Define
\beq \label{muzero}
\delta=1-\lambda\in [0,1),\quad \alpha_*=n(a-\delta)/(2-a)\quad \text{and}\quad \mu_0=\frac{a-\delta}{1-a}.
\eeq

%Throughout, we denote 
%To simplify the formulas, we denote
%$c_*=\max\{c_1,c_2,c_3,c_4\}$ with $c_1,c_2,c_3,c_4$ defined in Lemma \ref{gentrace}.

%==================================================================%
\begin{lemma}\label{newtrace}
Assume $a>\delta$, $\alpha\ge 2-\delta$ and $\alpha> n\mu_0$.
Let $c_*=\max\{c_1,c_2,c_3,c_4\}$ with $c_1,c_2,c_3,c_4$ in Lemma \ref{gentrace}, and 
\beq
\label{theta}
\theta=\theta_\alpha \eqdef \frac{1}{(1-a)(\alpha/\alpha_*-1)} \in (0,1).
\eeq
Then one has for any $\varepsilon>0$ that
 \beq\label{trace-thm1}
  \begin{split}
   \int_\Gamma |u|^\alpha d\sigma 
&\le 2\varepsilon \int_U |u|^{\alpha+\delta-2}|\nabla u|^{2-a} dx 
+ c_*\norm{u}_{L^\alpha(U)}^\alpha\\
 &\quad +  D_{3,\alpha}\varepsilon^{-\frac  1{1-a}} \norm{u}_{L^\alpha(U)}^{\alpha+\mu_0}
 +  D_{4,\alpha}\varepsilon^{-\midx_2}
 \norm{u}_{L^\alpha(U)}^{\alpha + \midx_1},
  \end{split}
 \eeq
where 
\begin{align}
\label{muteen}
\midx_1&=\midx_{1,\alpha} \eqdef \frac{\mu_0(1+\theta(1-a))}{1-\theta},
\quad
\midx_2=\midx_{2,\alpha}\eqdef \frac1{1-a}+\frac {\theta(2-a)}{(1-\theta)(1-a)} ,\\
\label{newD1} 
D_{3,\alpha}&= 2^{\theta(\alpha+\delta-a)}c_*^\frac{(2-a)(1+\theta(1-a))}{1-a}  \alpha^\frac{2-a}{1-a}|U|^{\frac{(1-a)(\alpha+\mu_0)\theta}{\alpha}},\\
\label{newD2} 
D_{4,\alpha}&= 2^\frac{\theta(\alpha+\delta-a)}{1-\theta}
(c_*\alpha)^\frac{(2-a)(1+\theta(1-a))}{(1-a)(1-\theta)}.
\end{align}
\end{lemma}
%==================================================================%
\begin{proof} 
We apply inequality \eqref{trace30} in Lemma \ref{gentrace} to $p=2-a$ and $s=2-\delta$. We recalculate exponents for these particular values.
Note, $\alpha-s+p=\alpha+\delta-a.$
Then from  \eqref{mdef} and \eqref{rdef},
\beqs
m=\frac{\alpha+\delta-a}{2-a}, 
\quad 
r=\frac{s-p}{p-1}=\frac{a-\delta}{1-a}=\mu_0.
%\quad q=\frac{n(\alpha+\delta-a)}{n-2+a}.
\eeqs
Also, $\theta$ in \eqref{thetaorig} becomes \eqref{theta}, the number $\midx_1$ in \eqref{mdef} is the same as in \eqref{muteen}.
The exponent of $|U|$ is $\theta(\alpha(1-a)+a-\delta)/\alpha=\theta(\alpha(1-a)+(1-a)\mu_0)/\alpha=(1-a)\theta(\alpha+\mu_0)/\alpha$.
The exponent of $\varepsilon^{-1}$ in the last term of \eqref{trace30} is $\midx_2$.
Also using the fact $c_1,c_2,c_3,c_4\le c_*$ and $m\le \alpha$, we  have $D_{1,\alpha}\le D_{3,\alpha}$ and $D_{2,\alpha}\le D_{4,\alpha}$.
Therefore, we obtain \eqref{trace-thm1} from \eqref{trace30}.
\end{proof}
%==================================================================%

Next is a parabolic multiplicative Sobolev inequality.

%=================================
\begin{lemma}\label{Parabwk}
Assume
\beq \label{alpcond}
\alpha\ge 2-\delta\quad{and}\quad \alpha > \alpha_*.
\eeq %(comes from $\kappa>1$)
If $T>0$, then
\beq\label{parabwki}
\begin{aligned}
\Big(\int_0^T \int_U |u|^{\kappa\alpha}dx dt \Big)^{\frac 1{\kappa\alpha}}
&\le (c_5\alpha^{2-a})^\frac1{\kappa\alpha} \Big(  \int_0^T\int_U |u|^{\alpha+\delta-a}dxdt+\int_0^T\int_U |u|^{\alpha+\delta-2}|\nabla u|^{2-a}dxdt \Big)^{\frac {\tilde \theta}{\alpha+\delta-a}} \\
&\quad \cdot\sup_{t\in[0,T]}\Big( \int_U|u(x,t)|^{\alpha}dx \Big)^{\frac{1-\tilde \theta}{\alpha}},
\end{aligned}
\eeq
where  $c_5\ge 1$ is independent of $\alpha$ and T, and 
\beq \label{kappadef}
\tilde \theta=\tilde \theta_\alpha \eqdef \frac 1{1+\frac{\alpha(2-a)}{n(\alpha+\delta-a)}},\quad 
\kappa=\kappa(\alpha)  \eqdef 1+\frac{2-a}n-\frac{a-\delta}{\alpha}=1+(a-\delta)(\frac1{\alpha_*}-\frac1{\alpha}).
\eeq

In case $U=B_R$ - a ball of radius $R$ - one has
\begin{multline}\label{paraball}
\Big(\int_0^T \int_{B_R} |u|^{\kappa\alpha}dx dt \Big)^{\frac 1{\kappa\alpha}}
\le  [c_6(1+R^{-1})^{2-a}\alpha^{2-a}]^\frac1{\kappa\alpha}\\
\cdot \Big[  \int_0^T\int_{B_R} |u|^{\alpha+\delta-a}dxdt
+\int_0^T\int_{B_R} |u|^{\alpha+\delta-2}|\nabla u|^{2-a}dxdt \Big]^{\frac {\tilde \theta}{\alpha+\delta-a}}\\
 \cdot \sup_{t\in[0,T]}\Big( \int_{B_R}|u(x,t)|^{\alpha}dx \Big)^{\frac{1-\tilde \theta}{\alpha}},
\end{multline}
where $c_6\ge 1$ independent of $\alpha$, $R$ and $T$.
\end{lemma}
%=================================
\begin{proof}
Note that by definition of $\kappa$ and $\tilde \theta$ we have $\frac 1{\kappa \alpha}=\frac {\tilde \theta}{p_0}+\frac{1-\tilde \theta}{\alpha}$, where $p_0=\frac{n(\alpha+\delta-a)}{n-(2-a)}$. Then interpolation inequality gives 
\beq\label{parestfmos}
\begin{aligned}
\Big( \int_U |u|^{\kappa \alpha}dx \Big)^{\frac 1{\kappa\alpha}}\le \Big(\int_U |u|^{p_0}dx\Big)^{\frac{\tilde \theta}{p_0}}\cdot \Big( \int_U|u|^{\alpha}dx \Big)^{\frac{1-\tilde \theta}{\alpha}}.
\end{aligned}
\eeq
We estimate the first integral on the right-hand side of \eqref{parestfmos}. We recall the following Sobolev's inequality
\begin{align*}
\|w\|_{L^{(2-a)^*}}\le c_7 \Big(\int_U |w|^{2-a}dx + \int_U |\nabla w|^{2-a}dx\Big)^\frac1{2-a},
\end{align*}
where $(2-a)^*$ is  the Sobolev conjugate of $2-a$, and $c_7\ge 1$ is independent of $\alpha$.
Applying this inequality to  $w=|u|^m$ with $m=\frac{\alpha+\delta-a}{2-a}\ge 1,$ we obtain  
\beq\label{uchangew}
 \Big(\int_U |u|^{\frac{n(\alpha+\delta-a)}{n-(2-a)}}dx\Big)^{\frac 1{(2-a)^*}}
\le c_7\Big(\int_U |u|^{\alpha+\delta-a}dx+m^{2-a}\int_U |u|^{\alpha+\delta-2}|\nabla u|^{2-a}dx\Big)^{\frac 1{2-a}}.
\eeq
Note that 
$\kappa \alpha\tilde \theta=\alpha+\delta-a$ and 
$(2-a)^*/p_0=(2-a)/(\alpha+\delta-a)$.
Then raising  both sides  of inequality \eqref{uchangew} by $(2-a)^*\kappa\alpha\tilde \theta/p_0=2-a$ results in
\beq\label{iter1}
\Big(\int_U |u|^{p_0}dx\Big)^{\frac {\kappa\alpha \tilde \theta}{p_0}}\le c_7^{2-a}\Big( \int_U |u|^{\alpha+\delta-a}dx+m^{2-a}\int_U |u|^{\alpha+\delta-2}|\nabla u|^{2-a}dx\Big).
\eeq
Raising  both sides of  \eqref{parestfmos}  to the power $\kappa\alpha$, and using inequality \eqref{iter1},  we get
\begin{align*}
\int_U |u|^{\kappa\alpha}dx 
&\le (c_7m)^{2-a}\Big( \int_U |u|^{\alpha+\delta-a}dx+m^{2-a}\int_U |u|^{\alpha+\delta-2}|\nabla u|^{2-a}dx\Big) \Big( \int_U|u|^{\alpha}dx \Big)^{(1-\tilde \theta)\kappa}.
\end{align*}

Integrating this inequality in $t$ from $0$ to $T$ and taking the supremum of the last integral  for $t\in[0,T]$ yield
\begin{align*}
\int_0^T \int_U |u|^{\kappa\alpha}dx dt 
&\le (c_7m)^{2-a} \Big( \int_0^T\int_U |u|^{\alpha+\delta-a}dxdt+\int_0^T\int_U |u|^{\alpha+\delta-2}|\nabla u|^{2-a}dxdt \Big) \\
&\quad \cdot\sup_{t\in[0,T]}\Big( \int_U|u(x,t)|^{\alpha}dx \Big)^{(1-\tilde \theta){\kappa}}.
\end{align*}
Taking  both sides to the power $\frac{1}{\kappa\alpha}=\frac {\tilde \theta}{\alpha+\delta-a}$, we have
\begin{align*}
\Big(\int_0^T \int_U |u|^{\kappa\alpha}dx dt \Big)^{\frac 1{\kappa\alpha}}
&\le (c_7m)^\frac{2-a}{\kappa \alpha} \Big(  \int_0^T\int_U |u|^{\alpha+\delta-a}dxdt+\int_0^T\int_U |u|^{\alpha+ \delta-2}|\nabla u|^{2-a}dxdt \Big)^\frac {\tilde \theta}{\alpha+\delta-a} \\
&\quad \cdot\sup_{t\in[0,T]}\Big( \int_U|u(x,t)|^{\alpha}dx \Big)^{\frac{1-\tilde \theta}{\alpha}}.
\end{align*}
Note that $m<\alpha$, then we obtain \eqref{parabwki} with $c_5=c_7^{2-a}$.

Consider the case $U=B_R$. The Sobolev inequality for $B_R$ is
\beq\label{BR}
\|w\|_{L^{(2-a)^*}(B_R)}
\le c_8 R^{-1}\Big(\int_{B_R} |w|^{2-a}dx\Big)^{\frac 1{2-a}}+c_8\Big(\int_{B_R} |\nabla w|^{2-a}dx\Big)^{\frac 1{2-a}},
\eeq
where $c_8\ge 1$ is independent of $R$. Repeating the above proof with $c_7$ replaced by $c_8(1+R^{-1})$, we obtain \eqref{parabwki} with 
$c_5$ replaced by $[c_8(1+R^{-1})]^{2-a}$. Therefore, we obtain \eqref{paraball} with $c_6=c_8^{2-a}$.
\end{proof}
%================================%

It is noteworthy that the explicit exponents and constants in Lemmas \ref{newtrace} and \ref{Parabwk} play an important role in Moser's iterations in sections \ref{Linterior} and \ref{Lglobal}.

%==================================================================%
%==================================================================%

%%%%%%%%%%%%%%%%%%%%%%%%%%%%%%%%%%%%%%%%%%%%%%%%%%%%%%%%%%%%%%%%%%%%%
\myclearpage    
\section{$L^\alpha$-estimates}\label{Lalpha}

We start studying problem \eqref{rho:eq}.
Hereafter, $u(x,t)$ denotes a non-negative solution of \eqref{rho:eq}.

Regarding the nonlinearity of the PDE in \eqref{rho:eq}, we recall that  the function $K(\xi)$  has the following properties: it is decreasing in $\xi$,  maps $[0,\infty)$ onto $(0,\frac{1}{a_0}]$,
 and
  \begin{align}
\frac{d_1}{(1+\xi)^a}\le K(\xi)\le \frac{d_2}{(1+\xi)^a},\label{i:ineq1} \\ 
  d_3(\xi^{2-a}-1)\le K(\xi)\xi^2\le d_2\xi^{2-a}, \label{i:ineq2}
%\\   -aK(\xi)\le K'(\xi)\xi\le 0, \label{i:ineq3}
  \end{align}
  where $d_1, d_2, d_3$ are positive constants depending on $g(s)$. (See \cite{ABHI1} for the proof.)

\textbf{Notation.} 
The symbol $C$ denotes a generic positive constant with varying values in different places, while $C_1,C_2,\ldots$ and $c_1,c_2,\ldots$ have their values fixed.
The constants $C$, $\hat C$, $C_i$, $c_j$, for $i,j=1,2,3,\ldots,$ used in calculations can, otherwise stated,  implicitly  depend on number $\lambda$, function $g$, the space dimension $n$, and $U$. 
We also use the positive and negative parts notation $f^+=\max\{f,0\}$ and $f^-=\max\{-f,0\}$, and brief notation $\|\varphi(t)\|_{L^\infty(\Gamma)}$ for $\|\varphi(\cdot,t)\|_{L^\infty(\Gamma)}$.

Recall $\delta$, $\alpha_*$ and $\mu_0$ are defined in \eqref{muzero}.
We assume hereafter that
\beq\label{ad}
a>\delta,\text{ i.e., the number $\mu_0$  is positive.}
\eeq

The case $a\le \delta$ is much simpler, see Remark \ref{asmall} below.
We begin with a differential inequality for $\int_U u^\alpha dx$.

%==========================================================%
\begin{lemma} \label{dalpha}
%\fbox{Assume  $1<2-a<n$, $\alpha\ge 2-\delta>2-a$, $\alpha> \frac{n(a-\delta)}{1-a}$, and $\alpha \ge \frac{a-\delta}{1-a}$.}
Assume  $\alpha\ge 2-\delta$ and $\alpha> n\mu_0$. 
Then
\begin{multline}\label{est10}
\frac{d }{dt}\int_U u^{\alpha}(x,t) dx +C_1\int_U  |\nabla u(x,t)|^{2-a}u^{\alpha+\delta-2}(x,t)dx
\le C\|  u(t)\|_{L^{\alpha}(U)}^{\alpha+\delta-2}
 + C \norm{\varphi^-(t)}_{L^\infty(\Gamma)} \norm{u(t)}_{L^\alpha(U)}^\alpha\\
+C\norm{\varphi^-(t)}_{L^\infty(\Gamma)}^{\frac  {2-a}{1-a}} \norm{u(t)}_{L^\alpha(U)}^{\alpha+\mu_0} +  C\norm{\varphi^-(t)}_{L^\infty(\Gamma)}^\frac{2-a}{(1-a)(1-\theta)}
 \norm{u(t)}_{L^\alpha(U)}^{\alpha + \midx_1},
\end{multline}
with $\midx_1$ and $\theta$ in Lemma \ref{newtrace}, where the positive constants $C_1=C_{1,\alpha}$ and $C>0$ depend on $\alpha$.
Consequently, there is $C_2=C_{2,\alpha}>0$ such that
\beq\label{rho:est0}
\frac{d }{dt}\int_U u^{\alpha} dx+C_1\int_U  |\nabla u|^{2-a}u^{\alpha+\delta-2}dx\le C_2 \Big( 1+\norm{\varphi^-(t)}_{L^\infty(\Gamma)} ^\frac{2-a}{(1-a)(1-\theta)} \Big) \Big(1+\int_U u^{\alpha} dx\Big)^{1+\midx_1/\alpha}.
\eeq
\end{lemma}
%==========================================================%
\begin{proof}
 Multiplying both sides of the first equation in \eqref{rho:eq} by $ u^{\alpha+\delta-1}$, integrating over domain $U$, and using integration by parts, we have
 \begin{align}\label{weakform2}
\frac{\lambda}{\alpha}\frac{d }{dt}\int_U u^{\alpha} dx&= \int_U \nabla\cdot \left(K(|\nabla u|)\nabla u \right) u^{\alpha+\delta-1} dx \nonumber\\
&= - (\alpha-\lambda)\int_U K(|\nabla u|)|\nabla u|^2 u^{\alpha+\delta-2}  dx 
+\int_{\Gamma} K(|\nabla u|)\frac {\partial u}{\partial \vec\nu}  u^{\alpha+\delta-1}   d\sigma \nonumber \\
&= - (\alpha-\lambda)\int_U K(|\nabla u|)|\nabla u|^2 u^{\alpha+\delta-2}  dx-\int_\Gamma u^{\alpha} \varphi d\sigma.
\end{align}

Using relation  $K(\xi)\xi^2\geq d_3(\xi^{2-a}-1) $ in \eqref{i:ineq2},  one has 
\begin{align}\label{term2}
- \int_U K(|\nabla u|)|\nabla u|^2 u^{\alpha+\delta-2}  dx 
&\leq -d_3\int_U  |\nabla u|^{2-a}u^{\alpha+\delta-2}dx+d_3\int_U  u^{\alpha+\delta-2}dx \nonumber\\
&\leq -d_3\int_U  |\nabla u|^{2-a}u^{\alpha+\delta-2}dx+C\|u\|_{L^\alpha}^{\alpha+\delta-2}.
\end{align}

Estimate the last integral of \eqref{weakform2} by using the trace theorem in Lemma \ref{newtrace}, we have  
\begin{multline}\label{term1}
-\int_\Gamma u^{\alpha} \varphi d\sigma 
\le \int_\Gamma u^{\alpha} \varphi^-  d\sigma 
\le \norm{\varphi^-}_{L^\infty(\Gamma)} \int_\Gamma u^{\alpha} d\sigma\\
\le \norm{\varphi^-}_{L^\infty(\Gamma)} \Big \{  2\varepsilon \int_U |u|^{\alpha+\delta-2}|\nabla u|^{2-a} dx 
+ C\norm{u}_{L^\alpha(U)}^\alpha+ C \varep^{-\frac1{1-a}}\norm{u}_{L^\alpha(U)}^{\alpha+\mu_0} +  C\varepsilon^{-\midx_2}
 \norm{u}_{L^\alpha(U)}^{\alpha + \midx_1} \Big\}.
\end{multline}

Combining \eqref{weakform2},\eqref{term2} and  \eqref{term1}, we have
\begin{multline}\label{est1}
\frac{\lambda}{\alpha}\frac{d }{dt}\int_U u^{\alpha} dx
\le  -(\alpha-\lambda)d_3\int_U  |\nabla u|^{2-a}u^{\alpha+\delta-2}dx +C\|u\|_{L^\alpha}^{\alpha+\delta-2}
 +\norm{\varphi^-}_{L^\infty(\Gamma)}\\
 \cdot  \Big\{ 2\varepsilon \int_U |u|^{\alpha+\delta-2}|\nabla u|^{2-a} dx 
+ C\norm{u}_{L^\alpha(U)}^\alpha+  C\varepsilon^{-\frac  1{1-a}} \norm{u}_{L^\alpha(U)}^{\alpha+\mu_0} +  C\varepsilon^{-\midx_2}
 \norm{u}_{L^\alpha(U)}^{\alpha + \midx_1} \Big\}.
\end{multline}%
The case $\norm{\varphi^-(t)}_{L^\infty(\Gamma)}=0$, inequality \eqref{est10} immediately follows \eqref{est1}.
Consider  $\norm{\varphi^-(t)}_{L^\infty(\Gamma)}\ne 0$. Select $\varepsilon = \frac {d_3} {4\norm{\varphi^-}_{L^\infty(\Gamma)}}$ in \eqref{est1}.  Then we obtain
\begin{multline}\label{est0}
\frac{\lambda}{\alpha}\frac{d }{dt}\int_U u^{\alpha} dx
\le  -(\alpha-\lambda-\frac 12)d_3\int_U  |\nabla u|^{2-a}u^{\alpha+\delta-2}dx 
+C\|  u\|_{L^{\alpha}(U)}^{\alpha+\delta-2} 
+C\norm{\varphi^-}_{L^\infty(\Gamma)} \norm{u}_{L^\alpha(U)}^\alpha
\\
+  C\norm{\varphi^-}_{L^\infty(\Gamma)}^{\frac  {2-a}{1-a}} \norm{u}_{L^\alpha(U)}^{\alpha+\mu_0} 
+  C\norm{\varphi^-}_{L^\infty(\Gamma)}^{\midx_2+1} \norm{u}_{L^\alpha(U)}^{\alpha + \midx_1}.
\end{multline}
Note that $\midx_2+1=\frac{2-a}{(1-a)(1-\theta)}$.
Multiplying both sides of \eqref{est0} by $\alpha/\lambda$, we get  inequality  \eqref{est10}.
%\beq
%\begin{aligned}
%&\frac{d }{dt}\int_U u^{\alpha} dx
%\le  -C_1\int_U  |\nabla u|^{2-a}u^{\alpha+\delta-2}dx +C\norm{\varphi^-}_{L^\infty(\Gamma)} \norm{u}_{L^\alpha(U)}^\alpha\\
%& \quad + C\|  u\|_{L^{\alpha}(U)}^{\alpha+\delta-2}+C\norm{\varphi^-}_{L^\infty(\Gamma)}^{\frac  {2-a}{1-a}} \norm{u}_{L^\alpha(U)}^{\alpha+\mu_0} +  C\norm{\varphi^-}_{L^\infty(\Gamma)}^{\midx_2+1}
% \norm{u}_{L^\alpha(U)}^{\alpha + \midx_1}.
%\end{aligned}
%\eeq
%Therefore, inequality  \eqref{est10} follows.

Now, we prove \eqref{rho:est0}. Denote $\beta= 1+\midx_1/\alpha$. 
Note that, on the right-hand side of \eqref{est10}, the maximum power of $\|u\|_{L^\alpha}$ is $\alpha+\midx_1=\alpha\beta$, and the maximum power of $\norm{\varphi^-}_{L^\infty(\Gamma)}$ is $\frac{2-a}{(1-a)(1-\theta)}$.
By applying inequality \eqref{ee5}, 
$$\|  u\|_{L^{\alpha}(U)}^{\alpha+\delta-2} ,\ 
\norm{u}_{L^\alpha(U)}^\alpha,\
 \norm{u}_{L^\alpha(U)}^{\alpha+\mu_0} 
\le 1+\|u\|_{L^\alpha}^{\alpha \beta},$$
$$\norm{\varphi^-}_{L^\infty(\Gamma)} ,\
\norm{\varphi^-}_{L^\infty(\Gamma)}^{\frac  {2-a}{1-a}}\le 1+\norm{\varphi^-}_{L^\infty(\Gamma)}^\frac{2-a}{(1-a)(1-\theta)}.$$
%each term $\|u\|_{L^\alpha}$ to a power on the right-hand side of \eqref{est10}  is bounded by $1+\|u\|_{L^\alpha}^{\alpha \beta}$, and each term $\norm{\varphi^-}_{L^\infty(\Gamma)}$ to a power in \eqref{est10} is bounded by $1+\norm{\varphi^-}_{L^\infty(\Gamma)}^{\midx_2+1}$.
Therefore, it follows from  \eqref{est10} that
\begin{multline*}
\frac{d }{dt}\int_U u^{\alpha} dx+C_1\int_U  |\nabla u|^{2-a}u^{\alpha+\delta-2}dx
\le  C(1+ \norm{\varphi^-}_{L^\infty(\Gamma)}^\frac{2-a}{(1-a)(1-\theta)})(1+\|u\|_{L^\alpha}^{\alpha\beta})\\
\le C(1+ \norm{\varphi^-}_{L^\infty(\Gamma)}^\frac{2-a}{(1-a)(1-\theta)})(1+\|u\|_{L^\alpha}^\alpha)^\beta.
\end{multline*}
Thus, inequality  \eqref{rho:est0} follows.   
\end{proof}
%=================================

\begin{remark}\label{asmall}
In case $a\le\delta$, by combining \eqref{weakform2} and the simple trace inequality \eqref{trace10}, we can find
\begin{align}\label{weakform3}
\frac{\lambda}{\alpha}\frac{d }{dt}\int_U u^{\alpha} dx&= \int_U \nabla\cdot \left(K(|\nabla u|)\nabla u \right) u^{\alpha+\delta-1} dx \nonumber\\
&= - C_1\int_U K(|\nabla u|)|\nabla u|^2 u^{\alpha+\delta-2}  dx 
+C_\varphi\Big( \int_U |u|^\alpha dx  + \int_U |u|^{\alpha+\mu_0} dx\Big),
\end{align}
where $C_\varphi>0$ depends on the function $\varphi$. Since $\mu_0\le0$ in this case, we can apply H\"older's inequality and easily derive a differential inequality for $\int_U u^{\alpha} dx$, and consequently obtain its estimates. Hence, there is no need to use more involved trace inequality \eqref{trace30}, which is essential in the proof of Lemma \ref{dalpha} for the case $a>\delta$. Therefore, we refer to $a>\delta$ as the super-critical case, and $a<\delta$ as the sub-critical case.
Note that papers \cite{Tsutsumi1988,ManfrediVespri} fall into the latter case.

\end{remark}

We now have local (in time) estimates for solutions.
%=================================
\begin{theorem}\label{est-sol}
Assume $\alpha\ge 2-\delta$ and $\alpha> n\mu_0$.
\begin{enumerate}
\item  If $T>0$ satisfies
\beq\label{T1}
\int_0^T (1+\norm{\varphi^-(t)}_{L^\infty(\Gamma)} ^\frac{2-a}{(1-a)(1-\theta)}) dt < \frac {1}{C_3}  \Big(1+\int_U u_0^{\alpha}(x) dx\Big)^{-\midx_1/\alpha},
\eeq
where  $C_3=C_{3,\alpha}\eqdef C_2\midx_1/\alpha>0$,
then for all $t\in(0,T]$:
\beq\label{rho:est1}
 \int_U u^{\alpha}(x,t) dx \le  \Big\{  \Big( 1+\int_U u_0^{\alpha}(x) dx\Big)^{-\midx_1/\alpha}  - C_3 \int_0^t (1+\norm{\varphi^-(\tau)}_{L^\infty(\Gamma)} ^\frac{2-a}{(1-a)(1-\theta)}) d\tau  \Big\}^{-\frac\alpha{\midx_1} }.
\eeq 

\item  Consequently, if $T>0$ satisfies 
\beq\label{TesT}
\int_0^T (1+\norm{\varphi^-(t)}_{L^\infty(\Gamma)} ^\frac{2-a}{(1-a)(1-\theta)}) dt \le \frac {1-2^{-\midx_1/\alpha}}{C_3}  \Big(1+\int_U u_0^{\alpha}(x) dx\Big)^{-\midx_1/\alpha},
\eeq
then 
\beq\label{rho:est2}
 \int_U u^{\alpha}(x,t) dx \le  2 \left(1+\int_U u_0^{\alpha}(x) dx\right)
\quad\text{for all }t\in(0,T],
\eeq
\beq\label{mixedterm}
 \int_0^T \int_U u^{\alpha+\delta-2}(x,t)|\nabla u(x,t)|^{2-a}dx dt
 \le C \Big (1+\int_U u_0^{\alpha}(x) dx\Big ),
\eeq
where  $C>0$ depends on $\alpha$.
\end{enumerate}
\end{theorem}
%=========================================%
\begin{proof}
(i) Let $V(t)= 1+\int_U u^{\alpha}(x,t) dx$. Then inequality \eqref{rho:est0} can be written as
\beq\label{Vest1}
\frac {d}{dt} V(t) \le C_2 (1+\norm{\varphi^-}_{L^\infty(\Gamma)} ^\frac{2-a}{(1-a)(1-\theta)}) V(t)^{1+\midx_1/\alpha}.
\eeq
Solving this differential inequality gives
\[
V(t)\le \Big\{V(0)^{-\midx_1/\alpha}  - \frac{\midx_1C_2}\alpha\int_0^t (1+\norm{\varphi^-(\tau)}_{L^\infty(\Gamma)} ^\frac{2-a}{(1-a)(1-\theta)}) d\tau\Big\}^{-\alpha/\midx_1 },
\] 
for all $t\in(0,T]$, with $T>0$ satisfying \eqref{T1}. Hence we have \eqref{rho:est1}. 

(ii) When $T>0$ satisfies \eqref{TesT}, inequality \eqref{rho:est2} easily follows \eqref{rho:est1}.

Integrating inequality \eqref{rho:est0} in time and using \eqref{rho:est2}, we have
\beqs
\begin{split}
&\int_0^T \int_U u^{\alpha+\delta-2}(x,t)|\nabla u(x,t)|^{2-a}dx dt \\
&\quad\le C \int_U u_0^{\alpha}(x) dx + C\int_0^T \Big (1+\int_U u^{\alpha}(x,t) dx\Big )^{1+\midx_1/\alpha}  \Big(1+\norm{\varphi^-(t)}_{L^\infty(\Gamma)}^\frac{2-a}{(1-a)(1-\theta)}\Big)dt.\\
&\quad\le C \int_U u_0^{\alpha}(x) dx +C\Big (1+\int_U u_0^{\alpha}(x) dx\Big )^{1+\midx_1/\alpha} \int_0^T \Big(1+\norm{\varphi^-(t)}_{L^\infty(\Gamma)}^\frac{2-a}{(1-a)(1-\theta)}\Big)dt.
\end{split}
\eeqs
Combining this with the bound \eqref{TesT} for the last integral, we obtain \eqref{mixedterm}.
\end{proof}

%====================================%

%%%%%%%%%%%%%%%%%%%%%%%%%%%%%%%%%%%%%%%%%%%%%%%%%%%%%%%%%%%%%%%%%%%%%
\myclearpage
\section{Gradient estimates} \label{GradSec}

%%%%%%%%%%%%%%%%%%%%%%%%%%%%%%%%%%%%%%%%%%%%%%%%%%%%%%%%%
In this section, we estimate $\int_U |\nabla u|^{2-a}(x,t)dx$ for $t>0$. Same as in \cite{ABHI1}, we will use the following function  
\beq\label{Hdef}
H(\xi)=\int_0^{\xi^2} K(\sqrt{s}) ds \quad \text{for } \xi\ge 0. 
\eeq
The function $H(\xi)$ can be compared with $\xi$ and $K(\xi)$ by
\beq
K(\xi)\xi^2 \le H(\xi)\le 2K(\xi)\xi^2, \label{i:ineq4}
\eeq
and hence, as a consequence of \eqref{i:ineq2} and \eqref{i:ineq4}, we have
\beq
d_3(\xi^{2-a}-1) \le H(\xi) \le 2 d_2\xi^{2-a}. \label{i:ineq5}
\eeq

%=================================%
\begin{proposition}\label{theo41}
Assume 
\beq\label{ac2} \alpha>\max\{ n\mu_0, \lambda+1+\mu_0\}.
\eeq
 If $t>0$ then 
\begin{multline}\label{nablauest}
\int_0^t\int_U u^{1-\lambda}|(u^{\lambda})_t|^2 dxd\tau +\int_U |\nabla u(x,t)|^{2-a} dx\\
\le C\mathcal E_0+C \mathcal{K}(t)
+ C\int_0^t \big(1+\norm{\varphi^-(\tau)}_{L^\infty(\Gamma)}^\frac{2-a}{(1-a)(1-\theta)}\big)
\big(1+\norm{u(\tau)}_{L^\alpha(U)}^{\alpha+\midx_1}\big)d\tau,
\end{multline}
where $C>0$ depends on $\alpha$,
\beq\label{Ezero}
\mathcal E_0 =\int_U u_0^{\alpha}(x) dx + \int_U |\nabla u_0(x)|^{2-a} dx + \int_\Gamma  u_0^{\lambda+1}(x) \varphi^+(x,0)  d\sigma,
\eeq
\beq\label{Kdeff}
\mathcal K(t)=1+\norm{\varphi^-(t)}_{L^\infty(\Gamma)}^{\midx_3 }+\int_0^t\int_\Gamma |\varphi_t(x,\tau)|^{\frac{\alpha}{\alpha-\lambda-1}}d\sigma d\tau,
\eeq
with
\beqs
\midx_3=\midx_{3,\alpha}\eqdef  \frac{(2-a)\alpha}{(1-a)(\alpha-(\lambda+1+\mu_0))}.
\eeqs
\end{proposition}
%========================================================%
\begin{proof} 
Multiplying both sides of the PDE in \eqref{rho:eq} by  
$u_t= \frac 1\lambda (u^\lambda)_t u^{1-\lambda}$, integrating over $U$ and using the boundary condition, we obtain
\beq\label{base}\begin{split}
\frac 1{\lambda} \int_U  u_t(u^{\lambda})_t dx +\frac 1 2 \frac d {dt}\int_U H(|\nabla u(x,t)|) dx  &= -\int_\Gamma  \varphi u^\lambda  u_t  d\sigma\\
&= -\frac 1 {\lambda+1}\frac d {dt} \int_\Gamma  u^{\lambda+1} \varphi  d\sigma + \frac 1 {\lambda+1}\int_\Gamma  u^{\lambda+1} \varphi_t d\sigma.
\end{split}
\eeq
Multiplying the equation by $\lambda+1$, and applying  Young's inequality to the last  integral yield
\begin{multline}\label{base2}
\frac{\lambda+1}{\lambda}\int_U u^{1-\lambda}[(u^{\lambda})_t]^2 dx +\frac d {dt}\Big(\frac{\lambda+1}2\int_U H(|\nabla u(x,t)|) dx + \int_\Gamma  u^{\lambda+1} \varphi  d\sigma\Big)= \int_\Gamma  u^{\lambda+1}\varphi_t d\sigma\\
\le  \int_\Gamma u^{\alpha} d\sigma+\int_\Gamma |\varphi_t|^{\frac{\alpha}{\alpha-\lambda-1}}d\sigma .
\end{multline}
To estimate the second to last boundary integral, we use the trace inequality \eqref{trace-thm1} in Lemma \ref{newtrace}:\begin{align*}
   \int_\Gamma |u|^\alpha d\sigma 
&\le 2\varepsilon \int_U |u|^{\alpha+\delta-2}|\nabla u|^{2-a} dx 
+ C\norm{u}_{L^\alpha(U)}^\alpha
+  C\varepsilon^{-\frac  1{1-a}} \norm{u}_{L^\alpha(U)}^{\alpha+\mu_0}
 +  C\varepsilon^{ - \midx_2 }
 \norm{u}_{L^\alpha(U)}^{\alpha+\midx_1},
\end{align*}
for any $\varepsilon>0$. Using this inequality in \eqref{base2},  we have 
\beq
\begin{aligned}\label{part1}
\frac{\lambda+1}{\lambda}\int_U u^{1-\lambda}(u^{\lambda})_t^2 dx &+\frac d {dt}\Big(\frac{\lambda+1}2\int_U H(|\nabla u(x,t)|) dx + \int_\Gamma  u^{\lambda+1} \varphi  d\sigma\Big)\\
&\le 2 \varepsilon \int_U |u|^{\alpha+\delta-2}|\nabla u|^{2-a} dx 
+ C \norm{u}_{L^\alpha(U)}^\alpha
+  C\varepsilon^{-\frac  1{1-a}} \norm{u}_{L^\alpha(U)}^{\alpha+\mu_0}\\
 &\quad +  C\varepsilon^{ - \midx_2 }
 \norm{u}_{L^\alpha(U)}^{\alpha+\midx_1}+\int_\Gamma |\varphi_t|^{\frac{\alpha}{\alpha-\lambda-1}}d\sigma.
\end{aligned}
\eeq
Define 
\beq\label{elamdef}
\mathcal E(t) = \frac{\lambda+1}2\int_U H(|\nabla u(x,t)|) dx + \int_\Gamma  u^{\lambda+1} \varphi  d\sigma+\int_U u^{\alpha} dx.
\eeq
Adding \eqref{est10} to \eqref{part1} yields
\begin{multline}\label{add110}
\frac{\lambda+1}{\lambda}\int_U u^{1-\lambda}(u^{\lambda})_t^2 dx +\frac d {dt}\mathcal E(t)+(C_1-2  \varepsilon) \int_U |u|^{\alpha+\delta-2}|\nabla u|^{2-a} dx \\
\le C(1+  \norm{\varphi^-}_{L^\infty(\Gamma)} )\norm{u}_{L^\alpha(U)}^\alpha+C\|u\|_{L^{\alpha}(U)}^{\alpha+\delta-2}
+  C (\varepsilon^{-\frac  1{1-a}}+  \norm{\varphi^-}_{L^\infty(\Gamma)}^{\frac  {2-a}{1-a}} ) \norm{u}_{L^\alpha(U)}^{\alpha+\mu_0}\\
 +  C( \varepsilon^{ - \midx_2 }+ \norm{\varphi^-}_{L^\infty(\Gamma)}^\frac{2-a}{(1-a)(1-\theta)})
 \norm{u}_{L^\alpha(U)}^{\alpha+\midx_1}+\int_\Gamma |\varphi_t|^{\frac{\alpha}{\alpha-\lambda-1}}d\sigma.
\end{multline}
Choosing $\varepsilon$ sufficiently small such that $C_1-2\varepsilon>0$, we derive 
\begin{multline*}
\frac{\lambda+1}{\lambda}\int_U u^{1-\lambda}(u^{\lambda})_t^2 dx +\frac d {dt}\mathcal E(t) 
\le  C(1+ \norm{\varphi^-}_{L^\infty(\Gamma)} )\norm{u}_{L^\alpha(U)}^\alpha
+C\|u\|_{L^{\alpha}(U)}^{\alpha+\delta-2}\\
+   C(1+\norm{\varphi^-}_{L^\infty(\Gamma)}^{\frac  {2-a}{1-a}} ) \norm{u}_{L^\alpha(U)}^{\alpha+\mu_0}
 + C(1+\norm{\varphi^-}_{L^\infty(\Gamma)}^\frac{2-a}{(1-a)(1-\theta)})
 \norm{u}_{L^\alpha(U)}^{\alpha+\midx_1}+\int_\Gamma |\varphi_t|^{\frac{\alpha}{\alpha-\lambda-1}}d\sigma.
\end{multline*}

Note that $ \alpha+\delta-2 <\alpha < \alpha+\mu_0 < \alpha+\midx_1$ and $1< \frac{2-a}{1-a}<\frac{2-a}{(1-a)(1-\theta)}$. We apply Young's inequality for each norm on the right-hand side and obtain
\begin{align*}
\frac{\lambda+1}{\lambda}\int_U u^{1-\lambda}(u^{\lambda})_t^2 dx +\frac d {dt}\mathcal E(t)\le C\int_\Gamma |\varphi_t|^{\frac{\alpha}{\alpha-\lambda-1}}d\sigma +  C( 1+ \norm{\varphi^-}_{L^\infty(\Gamma)}^\frac{2-a}{(1-a)(1-\theta)})\cdot ( 1+\norm{u}_{L^\alpha(U)}^{\alpha+\midx_1}).
\end{align*}
Let $t\in(0,T)$. Integrating both sides of previous inequality in $t$, we obtain 
\begin{align*}
&\frac{\lambda+1}{\lambda}\int_0^t\int_U u^{1-\lambda}(u^{\lambda})_t^2 dxd\tau +\frac{\lambda+1}2\int_U H(|\nabla u(x,t)|)dx+\int_U u^{\alpha} dx \\
&\le \mathcal E(0) -\int_\Gamma  u^{\lambda+1} \varphi  d\sigma
+C\int_0^t\int_\Gamma |\varphi_t|^{\frac{\alpha}{\alpha-\lambda-1}}d\sigma d\tau
+ C\int_0^t(1+\norm{\varphi^-}_{L^\infty(\Gamma)}^\frac{2-a}{(1-a)(1-\theta)})
 (1+\norm{u}_{L^\alpha(U)}^{\alpha+\midx_1})d\tau.
\end{align*}
For the first integral on the right-hand side, applying inequality \eqref{trace10} in Lemma \ref{gentrace} with $\alpha=s=\lambda+1=2-\delta$ and  $p=2-a$,  we have
\begin{align*}
- \int_\Gamma  u^{\lambda+1} \varphi  d\sigma 
&\le  \|\varphi^-\|_{L^\infty} \int_\Gamma u^{\lambda+1} d\sigma\\
&\le \|\varphi^-\|_{L^\infty}\Big\{ \varepsilon \int_U |\nabla u|^{2-a} dx +C\int_U |u|^{\lambda+1} dx  + C\varepsilon^{-\frac 1{1-a}}\int_U |u|^{\lambda+1+\mu_0} dx \Big\},
\end{align*}
for any $\varepsilon>0$.
  Now, using $ H(|\nabla u|) \ge C(|\nabla u|^{2-a} -1)$ from \eqref{i:ineq5} and applying Young's inequality to the last two integrals  with $\alpha> \lambda+1+\mu_0$ we obtain
\begin{align*}
&- \int_\Gamma  u^{\lambda+1} \varphi  d\sigma 
 \le C_8\varepsilon\|\varphi^-\|_{L^\infty(\Gamma)} \int_U (H(|\nabla u(x,t)|)+1)dx  \\
&\quad
+ (\frac14 \int_U u^{\alpha} dx+C \|\varphi^-\|_{L^\infty(\Gamma)}^\frac\alpha{\alpha-\lambda-1})
+(\frac14 \int_U u^{\alpha} dx+C\{\varepsilon^{-\frac 1{1-a}}\|\varphi^-\|_{L^\infty(\Gamma)}\}^\frac{\alpha}{\alpha-\lambda-1-\mu_0}).
\end{align*}
Selecting $\varep=\frac{\lambda+1}{4C_8(1+\|\varphi^-\|_{L^\infty} )}$, we obtain  
\beqs
\begin{split}
&\frac{\lambda+1}{\lambda}\int_0^t\int_U u^{1-\lambda}(u^{\lambda})_t^2 dxd\tau +\frac{\lambda+1}4\int_U H(|\nabla u(x,t)|) dx+\frac12\int_U u^{\alpha} dx \\
&\le \mathcal E(0)+C+ C\|\varphi^-\|_{L^\infty(\Gamma)}^\frac\alpha{\alpha-\lambda-1}
+C\{(1+\|\varphi^-\|_{L^\infty(\Gamma)})^\frac{1}{1-a}\|\varphi^-\|_{L^\infty(\Gamma)}\}^\frac{\alpha}{\alpha-\lambda-1-\mu_0}\\
&\quad +C\int_0^t\int_\Gamma |\varphi_t|^{\frac{\alpha}{\alpha-\lambda-1}}d\sigma d\tau
+ C\int_0^t (1+\norm{\varphi^-}_{L^\infty(\Gamma)}^\frac{2-a}{(1-a)(1-\theta)})
(1+\norm{u}_{L^\alpha(U)}^{\alpha+\midx_1})d\tau\\
\end{split}
\eeqs
which gives
\begin{multline}\label{prenabla}
\int_0^t\int_U u^{1-\lambda}(u^{\lambda})_t^2 dxd\tau +\int_U H(|\nabla u(x,t)|) dx\\
\le C \Big\{\mathcal E(0)+ \mathcal{K}_1(t)+\int_0^t  (1+\norm{\varphi^-}_{L^\infty(\Gamma)}^\frac{2-a}{(1-a)(1-\theta)})
(1+\norm{u}_{L^\alpha(U)}^{\alpha+\midx_1})d\tau\Big\},
\end{multline}
where 
\begin{align*}
\mathcal{K}_1(t)
&= 1+\|\varphi^-(t)\|_{L^\infty}^\frac{(2-a)\alpha}{(1-a)(\alpha-\lambda-1-\mu_0)}+ \|\varphi^-(t)\|_{L^\infty}^\frac\alpha{\alpha-\lambda-1}+ \int_0^t\int_\Gamma |\varphi_t|^{\frac{\alpha}{\alpha-\lambda-1}}d\sigma d\tau.
\end{align*}
Note that $\mathcal E(0)\le C\mathcal E_0$ and, by Young's inequality, $\mathcal{K}_1(t) \le C\mathcal K(t)$. Hence, we obtain from \eqref{prenabla}
\beq\label{nablauest2}
\begin{split}
&\int_0^t\int_U u^{1-\lambda}(u^{\lambda})_t^2 dxd\tau +\int_U H(|\nabla u(x,t)|) dx\\
&\le C\mathcal E_0+C \mathcal{K}(t)+ C\int_0^t (1+\norm{\varphi^-}_{L^\infty(\Gamma)}^\frac{2-a}{(1-a)(1-\theta)})
(1+\norm{u}_{L^\alpha(U)}^{\alpha+\midx_1})d\tau.
\end{split}
\eeq
Last,  using relation between $H(x,t)$ and $|\nabla u|$ again in \eqref{nablauest2}, we obtain \eqref{nablauest}.
% will derive
% \beqs
% \begin{split}
% \int_0^t\int_U u^{1-\lambda}(u^{\lambda})_t^2 dxd\tau +\int_U |\nabla u|^{2-a} dx \le C\mathcal{K}+ C\mathcal E(0)+C(1+\sup_{[0,t]} \norm{\varphi^-}_{L^\infty}^{\midx_2+1 })
% \int_0^t \norm{u}_{L^\alpha(U)}^{\alpha+\midx_1}d\tau\\
% \end{split}
% \eeqs
\end{proof}
%===============================================%

Now, we combine  Theorem~\ref{est-sol} and Proposition \ref{theo41} to obtain a  gradient  estimate in terms of initial and boundary data. 
%===============================================%
\begin{theorem}\label{theo42}
Assume \eqref{ac2}.
If $T>0$ satisfies \eqref{TesT}  then  for all $t\in(0,T]$, one has
\beq\label{estgradin2}
\int_0^t\int_U u^{1-\lambda}|(u^{\lambda})_t|^2 dxd\tau +\int_U |\nabla u(x,t)|^{2-a} dx 
\le C\mathcal E_0 + C\mathcal{K}(t),
\eeq
where $ \mathcal E_0 $ and $\mathcal{K}(t)$ are defined in Proposition \ref{theo41}, and $C>0$ depends on $\alpha$,
\end{theorem}
%============================================%
\begin{proof}
Using \eqref{rho:est2} to estimate the $L^{\alpha}$-norm of $u$ in \eqref{nablauest2}, we get 
\begin{multline*}
\int_0^t\int_U u^{1-\lambda}(u^{\lambda})_t^2 dxd\tau +\int_U |\nabla u|^{2-a} dx 
\le  C\mathcal E_0+C\mathcal{K}(t)\\+C\Big( 1+\int_U u_0^{\alpha}(x) dx\Big)^\frac{\alpha+\midx_1}\alpha
\int_0^t (1+\norm{\varphi^-(\tau)}_{L^\infty(\Gamma)} ^\frac{2-a}{(1-a)(1-\theta)}) d\tau .
\end{multline*}
By using \eqref{TesT} to bound the last integral, we obtain
\beqs%\label{estgradin2}
\int_0^t\int_U u^{1-\lambda}(u^{\lambda})_t^2 dxd\tau +\int_U |\nabla u|^{2-a} dx 
\le C\mathcal E_0 + C\mathcal{K}(t) + C\Big(1+\int_U u_0^{\alpha}(x) dx\Big).
\eeqs
Then  \eqref{estgradin2} follows.
\end{proof}
%====================================%

%%%%%%%%%%%%%%%%%%%%%%%%%%%%%%%%%%%%%%%%%%%%%%%%%%%%%%%%%%%%%%%%%%%%%
\myclearpage
\section{Interior $L^\infty$-estimates}\label{Linterior}
%%%%%%%%%%%%%%%%%%%%%%%%%%%%%%%%%%%%%%%%%%%%%%%%%%%%%%%%%%%%%%%%%%%%%

In this section, we  estimate the $L^\infty$-norm of the solution in the interior of the domain by using Moser's iteration. 
%We assume \eqref{ad}.

%%================================%
\begin{lemma}\label{interiorCac}
% \Enote{We do not use Poincare Sobolev therefore no need to assume following: Assume \eqref{alpcond}.}

Assume $\alpha\ge 2-\delta$. Suppose $B_R$ and $B_\rho$, with $R>\rho>0$, are two concentric balls  in a compact subset of $U$.
If $T>T_2>T_1\ge 0$ then 
\beq\label{Sest}
\begin{aligned}
&\sup_{t\in[T_2,T]} \int_{B_\rho}u^{\alpha}(x,t)dx +\int_{T_2}^T\int_{B_\rho} |\nabla u|^{2-a}u^{\alpha+\delta-2}dx dt\le  \mathcal C_\alpha\mathcal{S},
\end{aligned}
\eeq
where
\beq\label{Sdef}
\mathcal{S}=\int_{T_1}^T \int_{B_R} u^{\alpha}dx dt+\Big(\int_{T_1}^T \int_{B_R} u^{\alpha}dx dt\Big)^\frac{\alpha+\delta-2}{\alpha},
\eeq
\beq\label{Calpha}
\mathcal C_\alpha=\mathcal C_\alpha(\rho,R,T_1,T_2,T)\eqdef c_9\alpha^2 (1+|B_R|T) \Big(1+ \frac{1}{T_2-T_1}+\frac{1}{(R-\rho)^{2-a}}\Big),
\eeq 
with $c_9\ge 1$  independent of $U$, $\alpha$, $\rho$, $R$, $T_1$, $T_2$, $T$.
\end{lemma}
%===========================================%
\begin{proof}
%Define $Q_{R}=B_{R}\times [0,T]$ and $Q_{\rho}=B_{\rho}\times [0,T]$. 
Let  $\xi(x,t)=\xi_1(|x|)\xi_2(t)$ be the the cut-off function which is $1$ on $B_\rho\times(T_2,T)$ and has compact support in $B_R\times (T_1,T)$.
More specifically,  $\xi_1(|x|),\xi_2(t)\in[0,1]$ and satisfy 
\begin{align*}
\xi_1(|x|)=
\begin{cases}
1 \quad \text{if} \quad |x|<\rho,\\ 0 \quad \text{if} \quad |x|>R,
\end{cases}
\text{ and } \quad 
\xi_2(t)=
\begin{cases}
0 \quad \text{if} \quad 0\le t\le T_1,\\ 
1 \quad \text{if} \quad T_2<t<T.
\end{cases}
\end{align*}
Also, there is $C>0$ such that
\beq\label{cutder}
 |\xi_t|\le \frac C{T_2-T_1}\quad\text{and}\quad  |\nabla \xi|\le \frac C{R-\rho}.
\eeq 
In the calculations within this proof, notation $C$ denotes a generic constant independent of $\alpha$, $\rho$, $R$, $T_1$, $T_2$, $T$. 

Recall that $\delta=1-\lambda$.
Multiplying the PDE in \eqref{rho:eq} by $u^{\alpha+\delta-1}\xi^2$, integrating  over $U$, and using integration by parts, we obtain  
\begin{align*}
\lambda\int_U u^{\alpha-1}\xi^2\frac{\partial u}{\partial t}dx
&=\int_U \nabla \cdot (K (|\nabla u|)\nabla u  )u^{\alpha+\delta-1}\xi^2dx\\
&=-(\alpha-\lambda)\int_U K (|\nabla u|)|\nabla u|^2u^{\alpha+\delta-2}\xi^2dx-2\int_U K (|\nabla u|)\nabla u\cdot \nabla\xi u^{\alpha+\delta-1}\xi dx.
\end{align*}
Using properties  \eqref{i:ineq2}, resp., \eqref{i:ineq1} of function $K(\cdot)$ in the first, resp., second integral on the right-hand side of the last identity, we find
\beq\label{firstst}
\begin{aligned}
\lambda\int_U u^{\alpha-1}\xi^2\frac{\partial u}{\partial t}dx
&\le -d_3(\alpha-\lambda)\int_U |\nabla u|^{2-a}u^{\alpha+\delta-2}\xi^2dx+d_3(\alpha-\lambda)\int_Uu^{\alpha+\delta-2}\xi^2dx\\
&\quad+2d_2\int_U|\nabla u|^{1-a}|u|^{\alpha+\delta-1}\xi |\nabla\xi|dx.
\end{aligned}
\eeq
%
%where $d_3$ and $C$ are positive constants depending on $\vec{\alpha}$ and $\vec{a}$.

Let $\varepsilon>0$. Applying Young's inequality to the last integral of \eqref{firstst}, for conjugate exponents $\frac{2-a}{1-a}$ and $2-a$, we have 
\begin{align*}
\lambda\int_U \frac{\partial u^\alpha}{\partial t}\xi^2 dx&\le -d_3(\alpha-\lambda)\int_U |\nabla u|^{2-a}u^{\alpha+\delta-2}\xi^2dx+d_3(\alpha-\lambda)\int_Uu^{\alpha+\delta-2}\xi^2dx\\
&\quad+\varepsilon \int_U|\nabla u|^{2-a}|u|^{\alpha+\delta-2}\xi^2dx+C\varepsilon^{a-1}\int_U |u|^{\alpha+\delta-a}\xi^{a} |\nabla\xi|^{2-a}dx.
\end{align*}
Choosing $\varepsilon=\frac{d_3(\alpha-\lambda)}{2}$, we then have
\begin{multline*}
\frac{\lambda}{\alpha}\frac {d}{dt}\int_U u^\alpha\xi^2dx-\frac{2\lambda}{\alpha}\int_Uu^{\alpha}\xi \xi_tdx
\le -\frac {d_3(\alpha-\lambda)}{2}\int_U |\nabla u|^{2-a}u^{\alpha+\delta-2}\xi^2dx\\
\quad +C(\alpha-\lambda)\int_Uu^{\alpha+\delta-2}\xi^2dx
 +\frac{C}{(\alpha-\lambda)^{(1-a)}}\int_U |u|^{\alpha+\delta-a}\xi^{a} |\nabla\xi|^{2-a}dx.
\end{multline*}
Then integrating the inequality in time from $0$ to $t$ gives
\begin{multline}\label{xiwith}
\frac{\lambda}{\alpha} \int_U u^{\alpha}(x,t)\xi^2(x,t) dx
+\frac {d_3(\alpha-\lambda)}{2}\int_0^t\int_U |\nabla u|^{2-a}u^{\alpha+\delta-2}\xi^2dxd\tau\\
\le\frac{2\lambda}{\alpha}\int_0^t\int_U u^{\alpha}\xi |\xi_t|dxd\tau +C(\alpha-\lambda)\int_0^t\int_U u^{\alpha+\delta-2}\xi^2dxd\tau\\
+\frac{C}{(\alpha-\lambda)^{(1-a)}}\int_0^t\int_U |u|^{\alpha+\delta-a}\xi^{a} |\nabla\xi|^{2-a}dxd\tau.
\end{multline}
Using \eqref{cutder} to estimate $\xi_t$ and $\nabla \xi$ on the right-hand side of \eqref{xiwith}, the bound and support of $\xi(x,t)$,  we have 
\begin{multline*}
J\eqdef\frac{\lambda}{\alpha} \sup_{t\in[T_2,T]} \int_{B_\rho} u^{\alpha}(x,t)dx 
+\frac {d_3(\alpha-\lambda)}{2}\int_{T_2}^T\int_{B_\rho} |\nabla u|^{2-a}u^{\alpha+\delta-2}dxd\tau\\
\le\frac{\lambda C}{\alpha(T_2-T_1)}\int_{T_1}^T\int_{B_R} u^{\alpha}dxd\tau
  +C(\alpha-\lambda)\int_{T_1}^T\int_{B_R} u^{\alpha+\delta-2}dxd\tau\\
 +\frac{C}{(\alpha-\lambda)^{1-a}(R-\rho)^{2-a}}\int_{T_1}^T\int_{B_R} u^{\alpha+\delta-a}dxd\tau.
\end{multline*}
 Note that $\lambda\le 1$ and $\alpha-\lambda\ge 1$. We then have
\begin{align*}
J\le\frac{C}{T_2-T_1}\int_{T_1}^T\int_{B_R} u^{\alpha}dxd\tau +C(\alpha-\lambda)\int_{T_1}^T\int_{B_R} u^{\alpha+\delta-2}dxd\tau+\frac{C}{(R-\rho)^{2-a}}\int_{T_1}^T\int_{B_R}|u|^{\alpha+\delta-a}dxd\tau.
\end{align*}
Inequality \eqref{ee4} gives $u^{\alpha+\delta-a} \le u^{\alpha+\delta-2}+u^\alpha$, hence 
\begin{align*}
J&\le C(\alpha-\lambda)\Big[1+ \frac{1}{T_2-T_1}+\frac{1}{(R-\rho)^{2-a}}\Big]\cdot J_1,
\end{align*}
where $J_1=\int_{T_1}^T\int_{B_R} u^{\alpha}dxd\tau + \int_{T_1}^T\int_{B_R} u^{\alpha+\delta-2}dxd\tau$.
Therefore,
\beq\label{supual1}
\sup_{t\in[T_2,T]} \int_{B_\rho} u^{\alpha}(x,t)dx 
\le C\alpha^2 \Big[1+ \frac{1}{T_2-T_1}+\frac{1}{(R-\rho)^{2-a}}\Big]\cdot J_1,
\eeq
\beq\label{supual2}
\int_{T_2}^T\int_{B_\rho} |\nabla u|^{2-a}u^{\alpha+\delta-2}dxd\tau
\le C\Big[1+ \frac{1}{T_2-T_1}+\frac{1}{(R-\rho)^{2-a}}\Big]\cdot J_1.
\eeq
Applying  H\"older's inequality to the second integral of $J_1$ using exponents $\frac{\alpha}{\alpha+\delta-2}$ and $\frac{\alpha}{2-\delta}$ we have 
\beq \label{J1}
J_1 \le (1+(|B_R|T)^\frac{2-\delta}{\alpha})\mathcal{S} \le 2(1+|B_R|T)\mathcal S.
\eeq
Hence estimate  \eqref{Sest} follows from \eqref{supual1}, \eqref{supual2} and \eqref{J1}.
\end{proof}
%================================================================%

%================================================================%
\begin{proposition}\label{preMoser}
Assume \eqref{alpcond}. Let $B_R$, $B_\rho$, and $T$, $T_2$, $T_1$ and ${\mathcal C}_\alpha$ be as in Lemma \ref{interiorCac}. Then 
\beq\label{bfinest}
\| u\|_{L^{\kappa\alpha}(B_\rho\times (T_2,T))}
\le A_\alpha^\frac 1{\alpha}\Big( \| u\|_{L^\alpha(B_R\times(T_1,T))}^r+\| u\|_{L^\alpha(B_R\times(T_1,T))}^s\Big)^\frac 1{\alpha},
\eeq 
where $\kappa$ is defined in \eqref{kappadef},
\beq\label{rs}
r=r(\alpha)\eqdef \alpha+\delta-2,\quad s= s(\alpha) \eqdef \frac{\alpha^2}{\alpha+\delta-a},
\eeq
\beq\label{Aal}
A_{\alpha}=c_{10}(1+\rho^{-1})^{2-a}\alpha^{6-a} (1+|B_R|T)^2 \Big(1+ \frac{1}{T_2-T_1}+\frac{1}{(R-\rho)^{2-a}}\Big)^2,
\eeq
with $c_{10}\ge 1$ independent of $U$, $\alpha$, $\rho$, $R$, $T_1$, $T_2$, $T$.
\end{proposition}
%================================================================%
\begin{proof}
Applying Sobolev inequality \eqref{paraball} to the ball $B_\rho$ in place of $B_R$, we have
\beq\label{beginest}
J\eqdef \Big(\int_{T_2}^T \int_{B_\rho} |u|^{\kappa\alpha}dx dt \Big)^{\frac 1{\kappa\alpha}}
\le \hat c^\frac1{\kappa\alpha} \Big\{I^{\tilde \theta} \cdot\sup_{t\in[T_2,T]}\Big( \int_{B_\rho}|u|^{\alpha}dx \Big)^{1-\tilde \theta}\Big\}^\frac 1{\alpha},
\eeq
where $\hat c=c_6(1+\rho^{-1})^{2-a}\alpha^{2-a}$, exponent $\tilde \theta$ is defined in \eqref{kappadef}, and
\beqs
I=\Big( \int_{T_2}^T\int_{B_\rho} |u|^{\alpha+\delta-a}dxdt+\int_{T_2}^T\int_{B_\rho} |u|^{\alpha+\delta-2}|\nabla u|^{2-a}dxdt \Big)^{\frac {\alpha}{\alpha+\delta-a}}. 
\eeqs
Using \eqref{ee1}, we have
\beqs
I
\le C_4 \Big(\int_{T_2}^T\int_{B_\rho} |u|^{\alpha+\delta-a}dxdt\Big)^{\frac {\alpha}{\alpha+\delta-a}}+C_4\Big(\int_{T_2}^T\int_{B_\rho} |u|^{\alpha+\delta-2}|\nabla u|^{2-a}dxdt \Big)^{\frac {\alpha}{\alpha+\delta-a}},
\eeqs
 where $ C_4=2^{\frac{\alpha}{\alpha+\delta-a}}$.
%  \beq\label{C1def}
%  C_4=2^{\frac{\alpha}{\alpha+\delta-a}}.
%  \eeq
Then applying H\"older's inequality to the first integral on the right-hand side yields
\beq\label{Iest}
I
\le C_4 (T|B_\rho|)^\frac{a-\delta}{\alpha+\delta-a}\int_{T_2}^T\int_{B_\rho} |u|^{\alpha}dxdt+C_4\Big(\int_{T_2}^T\int_{B_\rho} |u|^{\alpha+\delta-2}|\nabla u|^{2-a}dxdt \Big)^{\frac {\alpha}{\alpha+\delta-a}}.
\eeq
Estimating the second integral on the right-hand side of \eqref{Iest} by \eqref{Sest}, and combining with \eqref{beginest} give
\begin{align*}
J
&\le \hat c^\frac1{\kappa\alpha} \Big\{  \Big( C_4 (T|B_\rho|)^\frac{a-\delta}{\alpha+\delta-a}\mathcal{S}+C_4({\mathcal C}_\alpha\mathcal{S})^{\frac{\alpha}{\alpha+\delta-a}}\Big)^{\tilde \theta}({\mathcal C}_\alpha\mathcal{S} )^{1-\tilde \theta} \Big\}^{\frac 1{\alpha}}\\
&\le \hat c^\frac1{\kappa\alpha} \Big\{  \Big( C_4 (T|B_\rho|)^\frac{a-\delta}{\alpha+\delta-a}\mathcal{S}+C_4({\mathcal C}_\alpha\mathcal{S})^{\frac{\alpha}{\alpha+\delta-a}}+{\mathcal C}_\alpha\mathcal{S} \Big)^{\tilde \theta+1-\tilde \theta} \Big\}^{\frac 1{\alpha}}\\
&\le \hat c^\frac1{\kappa\alpha} \Big\{ \big(  C_4 (T|B_\rho|)^\frac{a-\delta}{\alpha+\delta-a}+{\mathcal C}_\alpha\big )\mathcal{S}+C_4{\mathcal C}_\alpha^{\frac{\alpha}{\alpha+\delta-a}}\mathcal{S}^{\frac{\alpha}{\alpha+\delta-a}} \Big\}^{\frac 1{\alpha}}.
\end{align*}
Note by definition \eqref{Sdef} that $\mathcal S = y_{\alpha}^{\alpha}+y_{\alpha}^{\alpha+\delta-2}$, 
where
$y_{\alpha}=\Big(\int_{T_1}^T \int_{B_R} |u|^{\alpha}dx dt \Big)^{\frac 1{\alpha}}.$
Thus, we find that
\begin{align*}
J
&\le \hat c^\frac1{\kappa\alpha} \Big\{ (  C_4 (T|B_\rho|)^\frac{a-\delta}{\alpha+\delta-a}+{\mathcal C}_\alpha)(y_{\alpha}^{\alpha}+y_{\alpha}^{\alpha+\delta-2})+C_4{\mathcal C}_\alpha^{\frac{\alpha}{\alpha+\delta-a}}(y_{\alpha}^{\alpha}+y_{\alpha}^{\alpha+\delta-2})^{\frac{\alpha}{\alpha+\delta-a}} \Big\}^{\frac 1{\alpha}}\\
&\le \hat c^\frac1{\kappa\alpha}\Big\{ (M_1y_{\alpha}^{\alpha}+M_1y_{\alpha}^{\alpha+\delta-2}+M_2y_{\alpha}^\frac{\alpha^2}{\alpha+\delta-a}+M_2y_{\alpha}^\frac{(\alpha+\delta-2)\alpha}{\alpha+\delta-a}\Big\}^{\frac 1{\alpha}}
\end{align*}
where 
\beq\label{A12} 
M_1=C_4 (T|B_\rho|)^\frac{a-\delta}{\alpha+\delta-a}+C_{\alpha},\quad 
M_2=C_4{\mathcal C}_\alpha^{\frac{\alpha}{\alpha+\delta-a}}2^{\frac{\alpha}{\alpha+\delta-a}}=C_4^2C_{\alpha}^{\frac{\alpha}{\alpha+\delta-a}}.\eeq

Note that $\alpha + \delta -2 <  \frac{(\alpha+\delta-2)\alpha}{\alpha+\delta-a}<\alpha <\frac{\alpha^2}{\alpha+\delta-a} $, then by \eqref{ee4} we have 
\beqs
y_{\alpha}^{\alpha},\ y_{\alpha}^\frac{(\alpha+\delta-2)\alpha}{\alpha+\delta-a}\le y_{\alpha}^{\alpha+\delta-2}+y_{\alpha}^\frac{\alpha^2}{\alpha+\delta-a}.
\eeqs
Therefore, we obtain 
\beq\label{Jest}
J\le [3\hat c^\frac1\kappa(M_1+M_2)]^{\frac 1{\alpha}} (  y_{\alpha}^r+y_{\alpha}^s)^{\frac 1{\alpha}}.
\eeq 
Because $\frac{a-\delta}{\alpha+\delta-a}\le 1$, $\frac{\alpha}{\alpha+\delta-a}\le 2$, ${\mathcal C}_\alpha\ge 1+|B_R|T>1,$ and $1<C_4<4$, we have 
\beq\label{simp}
M_1+M_2\le C_4 (1+|B_R|T)+{\mathcal C}_\alpha +C_4^2{\mathcal C}_\alpha^2\le C_4 {\mathcal C}_\alpha +{\mathcal C}_\alpha +C_4^2{\mathcal C}_\alpha^2 \le 3 C_4^2{\mathcal C}_\alpha^2\le 3(4{\mathcal C}_\alpha)^2 .
\eeq
Also, $\hat c^\frac1\kappa\le \hat c$.
 Combining \eqref{Jest} and \eqref{simp}, we obtain
 \beqs
J\le [12^2\hat c\,  {\mathcal C}_\alpha^2]^{\frac 1{\alpha}} (  y_{\alpha}^r+y_{\alpha}^s)^{\frac 1{\alpha}}.
\eeqs 
Then  \eqref{bfinest} follows.
\end{proof}
%=========================================================%

% \begin{remark}\label{InterInq}
% \textcolor{blue}{Above, we used the following Young's/interpolation inequality. If $0<r<p<s$ then there is $\theta\in (0,1)$, such that  $p=\theta r +(1-\theta)s$.  For any $x>0$, 
% \beqs x^p = x^{\theta r} x^{(1-\theta)s} \le \theta x^r + (1-\theta)x^s\le x^r +x^s, 
% \eeqs
% where $c_1,c_2\in (0,1) $. Here we have used the Young's inequality for the power $\frac 1{\theta}, \frac 1{1-\theta}$. Combining $\alpha + \delta -2 <  \frac{(\alpha+\delta-2)\alpha}{\alpha+\delta-a}<\alpha <\frac{\alpha^2}{\alpha+\delta-a} $ and above note we find the following result.} 
% \end{remark}

Iterating relation \eqref{bfinest}, we obtain the following local estimate for $u$. 
%=========================================================%
\begin{theorem} \label{Linf1} 
Assume $\alpha_0>0$ such that $\alpha=\alpha_0$ satisfies \eqref{alpcond}.
Let $B_R$, with $R>0$,  be a ball in a compact subset of $U$, and $T>0$, $\sigma \in (0,1)$. Then
\beq\label{supualR}
\|u\|_{L^{\infty}(B_{R/2}\times(\sigma T,T))}\le \widehat{\mathcal C}_{R,T,\sigma}\max\Big\{ \|u\|^{\mu}_{L^{\alpha_0}(B_R\times(0,T))},\|u\|^{\nu}_{L^{\alpha_0}(B_R\times(0,T))}\Big\}, 
\eeq
where 
\beq\label{munu}
\mu=\prod_{j=0}^\infty \frac{\alpha_0 \kappa_*^j-2+\delta}{\alpha_0 \kappa_*^j},\quad \nu=\prod_{j=0}^\infty \frac{\alpha_0\kappa_*^j}{\alpha_0\kappa_*^j+\delta-a},
\eeq
 \beq\label{Cdef}
\widehat{\mathcal C}_{R,T,\sigma }=\Big[2^{11} c_{10} \alpha_0^{6-a} (1+R^{-1})^2 (1+|B_R|T)^3\Big(1+ \frac{1}{\sigma T} +\frac {1}{R^{2-a}}\Big)^2\Big]^\omega,
\eeq
with  $\kappa_*=\kappa(\alpha_0)$ defined in \eqref{kappadef}, 
constant $c_{10}$ as in Proposition \ref{preMoser}, and some positive number $\omega$ depending on $\alpha_0$.
\end{theorem}
%=========================================================%
\begin{proof}
For $j=0,1,2,\ldots $, let 
\beqs 
t_j=\sigma T(1-\frac 1{2^j}),\  
\rho_j=\frac {R}2(1+\frac 1{2^j}),\ 
Q_j=B_{\rho_j}\times(t_j,T),
\eeqs 
where $B_{\rho_j}$ is the ball of radius $\rho_j$ having the same center as $B_R$.
Note
\beqs 
t_j-t_{j+1}=\frac{\sigma T}{2^{j+1}}, \quad \rho_j-\rho_{j+1}=\frac R{2^{j+2}}, \quad
\lim_{j\to\infty}t_j=\sigma T,\quad
\lim_{t\to\infty} \rho_j=R/2.
\eeqs
Also, $\kappa_*>1$.  Let ${\alpha}_j={\alpha}_0\kappa_*^{j}$. Then $\alpha_j\ge \alpha_0$ gives $\kappa(\alpha_j)\ge \kappa(\alpha_0)=\kappa_*$.

Define $Y_j=\| u\|_{L^{\alpha_j}(Q_j)}.$ Note that $(Q_j)_{j=0}^\infty$ is a sequence of nested cylinders.
By H\"older's inequality we obtain  
\beqs
\| u\|_{L^{\kappa_*\alpha_j}(Q_{j+1})}\le \| u\|_{L^{\kappa(\alpha_j) \alpha_j}(Q_{j+1})}|Q_{j+1}|^{\frac1{\kappa_*\alpha_j}-\frac1{\kappa(\alpha_j) \alpha_j}}
\le (1+|Q_0|)^\frac1{\kappa_*\alpha_j} \| u\|_{L^{\kappa(\alpha_j) \alpha_j}(Q_{j+1})}.
\eeqs
Hence applying  \eqref{bfinest} to $\alpha=\alpha_j$, $\rho=\rho_{j+1}$, $R=\rho_j$, $T_2=t_{j+1}$ and $T_1=t_j$ gives
\begin{align*}
Y_{j+1} \le (1+|Q_0|)^\frac1{\kappa_*\alpha_j} A_{\alpha_j}^{\frac 1{\alpha_j}}\Big[Y_j^{r(\alpha_j)} +Y_j^{s(\alpha_j)}\Big]^{\frac 1{\alpha_j}}.
\end{align*}
Using definitions in \eqref{rs} and \eqref{Aal}, we denote
$$r_j=r(\alpha_j)=\alpha_j-2+\delta=\alpha_0 \kappa_*^j-2+\delta,$$
$$s_j=s(\alpha_j)=\alpha_j^2/(\alpha_j+\delta-a)=\alpha_0^2\kappa_*^{2j}/(\alpha_0\kappa_*^j+\delta-a),$$ 
%$$\kappa_j=\alpha_0 \kappa_*^j,$$
\beq\label{hatAj} 
\widehat A_j=(1+|Q_0|)^\frac1{\kappa_*} A_{\alpha_j}.
\eeq
Then
\beq\label{Moserform1}
Y_{j+1} \le \widehat A_j^{\frac 1{\alpha_j}}\big(Y_j^{r_j} +Y_j^{s_j}\big)^{\frac 1{\alpha_j}}.
\eeq

We estimate $\widehat A_j$. We have from \eqref{Aal} that
\begin{align*}
\widehat A_j&\le c_{10}(1+|Q_0|) \cdot (1+2/R)^2\alpha_j^{6-a} (1+|Q_0|)^2 \Big\{1+ \frac{2^{j+1}}{\sigma T} +\Big(\frac {2^{j+2}}{R}\Big)^{2-a}\Big\}^2 \\
&\le 4c_{10}(1+R^{-1})^2(1+|Q_0|)^3(\alpha_0 \kappa_*^j)^{6-a} 16^{j+2} (1+ \frac{1}{\sigma T} +\frac {1}{R^{2-a}})^2
\le A_{R,T,\sigma}^{j+1},
\end{align*}
where
\beqs
A_{R,T,\sigma}=\max\Big\{16\kappa_*^{6-a}, 4^5 c_{10} \alpha_0^{6-a} (1+R^{-1})^2 (1+|Q_0|)^3\Big(1+ \frac{1}{\sigma T} +\frac {1}{R^{2-a}}\Big)^2   \Big\}.
\eeqs
Since $\kappa_*\in(1,2)$, we actually have
\beq\label{Abound}
A_{R,T,\sigma}=4^5 c_{10} \alpha_0^{6-a} (1+R^{-1})^2 (1+|Q_0|)^3\Big(1+ \frac{1}{\sigma T} +\frac {1}{R^{2-a}}\Big)^2\ge 1.
\eeq
Therefore,
\beq\label{Yjstep}
Y_{j+1} \le A_{R,T,\sigma}^\frac{j+ 1}{\alpha_j}\Big[Y_j^{r_j} +Y_j^{s_j}\Big]^{\frac 1{\alpha_j}}\quad \forall j\ge 0.
\eeq
Since $\kappa_*>1$, we clearly have  $\sum_{j=0}^\infty (j+1)/\alpha_j$ converges to a positive number.
Note also that
\beq\label{cv0}
0<  \sum_{j=0}^\infty \ln \frac{s_j}{\alpha_j} = \sum_{j=0}^\infty \ln(1+\frac{a-\delta}{\alpha_0\kappa_*^j+\delta-a})\le \sum_{j=0}^\infty \frac{a-\delta}{\alpha_0\kappa_*^j+\delta-a}<\infty,
\eeq
and 
\beq\label{converg}
0< -\sum_{j=0}^\infty  \ln \frac{r_j}{\alpha_j}
=\sum_{j=0}^\infty  \ln \frac{\alpha_j}{r_j}
=\sum_{j=0}^\infty \ln(1+\frac {2+\delta}{\alpha_0 \kappa_*^j-2-\delta})
\le \sum_{j=0}^\infty \frac {2+\delta}{\alpha_0 \kappa_*^j-2-\delta}<\infty.
\eeq
Therefore, $\Pi_{j=0}^\infty (r_j/\alpha_j)$ and $\Pi_{j=0}^\infty (s_j/\alpha_j)$  converge to positive numbers $\mu$ and $\nu$, resp., given by \eqref{munu}.

By \eqref{Yjstep}, and applying Lemma \ref{Genn} to sequence $(Y_j)_{j=0}^\infty$, we obtain
\beqs
\|u\|_{L^\infty(B_{R/2}\times (\sigma T,T))}=\lim_{j\to\infty}Y_j\le (2A_{R,T,\sigma})^\omega \max\{Y_0^\mu,Y_0^\nu\},
\eeqs
for some positive number $\omega$. 
Then the desired estimate \eqref{supualR} follows.
\end{proof}
%======================================================%

\begin{remark}
Inequality \eqref{supualR} obviously leads to  the quasi-homogeneous estimate \eqref{Mishaineq}, which was proved in \cite{Surnachev2012} for equation
\beq\label{Mishaeq}
u_t=\nabla\cdot \mathbf A(x,t,u,\nabla u)
\eeq
with the homogeneous structure
\beq\label{Mishacond}
\mathbf A(x,t,u,\nabla u)\cdot \nabla u \ge c|u|^{m-1}|\nabla u|^p, \quad |\mathbf A(x,t,u,\nabla u)|\le c' |u|^{m-1}|\nabla u|^{p-1}.
\eeq

Due to the non-homogeneity of function $K(\cdot)$, see \eqref{i:ineq1}, our equation \eqref{maineq} cannot be converted to \eqref{Mishaeq}, \eqref{Mishacond}.
Therefore, above inequality \eqref{supualR} is an extension of \eqref{Mishaineq} to the class of equations \eqref{maineq} with non-homogeneous structure \eqref{i:ineq1}.
%This is in the same spirit as improvements in \cite{Surnachev2012}, and Proposition 3.2 (without $\delta$-terms) of \cite{HKP1}. This improves previous estimates when a positive constant is added to the right-hand side of \eqref{Mishaineq}. 
\end{remark}

%======================================================%

Now, we bound the $L^\infty$-norm of $u$, in any compact subsets of $U$, in terms of the initial and boundary data.

%======================================================%
\begin{theorem}\label{Linfcompact}
Let $U'$ be an open, relatively compact subset of $U$, and $\alpha=\alpha_0$ satisfy \eqref{alpcond}. Then for  $T >0$, and  $0<\varepsilon<\min\{1,T\}$, one has 
\beq \label{linfest5}
\|u\|_{L^{\infty}(U'\times(\varepsilon,T))}\le C \varep^{-2\omega}(1+T)^{3\omega}\max\Big\{ \|u\|_{L^{\alpha_0}(U\times(0,T))}^\mu, \|u\|_{L^{\alpha_0}(U\times(0,T))}^\nu \Big\}, 
\eeq
where $\omega,\mu,\nu$ are the same as in Theorem \ref{Linf1}.

In particular, if $T>0$ satisfies \eqref{T1} for $\alpha=\alpha_0$, then 
\beq \label{linfest4}
\|u\|_{L^{\infty}(U'\times(\varepsilon,T))}\le C \varep^{-2\omega}(1+T)^{3\omega}\max\Big\{ \Big(\int_0^{T}\mathcal{U}_{\alpha_0}(t) dt\Big)^{\frac{\mu}{\alpha_0}}, \Big(\int_0^{T}\mathcal{U}_{\varphi,\alpha_0}(t)dt\Big)^{\frac{\nu}{\alpha_0}} \Big\}, 
\eeq
where 
\beq\label{phidef}
\mathcal{U}_{\alpha_0}(t)= 
 \Big\{ \Big( 1+\int_U u_0^{\alpha_0}(x) dx\Big)^{-\frac{\midx_{1,\alpha_0}}{\alpha_0} }
 - C_{3,\alpha_0} \int_0^t (1+\norm{\varphi^-(\tau)}_{L^\infty(\Gamma)} )^{\frac  {2-a}{(1-a)(1-\theta_{\alpha_0})} } d\tau  \Big\}^{-\frac{\alpha_0}{\midx_{1,\alpha_0}} }.
\eeq

Furthermore, if $T>0$ satisfies \eqref{TesT} for $\alpha=\alpha_0$ then
\beq \label{linfest6}
\|u\|_{L^{\infty}(U'\times(\varepsilon,T))}\le C \varep^{-2\omega}(1+T)^{3\omega+\nu/\alpha_0}(1+\|u_0\|_{L^{\alpha_0}(U)})^\nu. 
\eeq

Above, constant $C>0$ depends on $U$, $U'$ and $\alpha_0$.
\end{theorem}
\begin{proof}
Denote $R=\frac 12 dist( \bar U', \partial U )>0$.
Because the set $\bar U'$ is compact, there exists finitely many $x_i\in \bar U',$ $i=1,\ldots,m$ for some $m\in \N$ such that $\{B_{R/2}(x_i)\}_{i=1}^m$ is an open covering of $\bar U'$.
For each $i$, we have from \eqref{supualR} with $\varep =\sigma T$ that 
\begin{align}
\|u\|_{L^{\infty}(B_{R/2}(x_i)\times(\varepsilon,T))}
&\le C_{R,T,\varepsilon}\max\Big\{ \|u\|^{\mu}_{L^{\alpha_0}(U\times(0,T))},\|u\|^{\nu}_{L^{\alpha_0}(U\times(0,T))}\Big\}, \label{ti3}
\end{align}
where
 \beqs
C_{R,T,\varep}=\Big\{2^{11}c_{10}\alpha_0^{6-a} (1+|B_R|T)^3 (1+R^{-1})^2(1+ \varep^{-1} +R^{a-2})^2   \Big\}^\omega
\le C_R\varep^{-2\omega}(1+T)^{3\omega}
\eeqs
for some positive number $C_R$ depending on $R$ and $\alpha_0$. Summing up the estimates \eqref{ti3} in $i$, we obtain \eqref{linfest5}.

In case $T>0$ satisfies \eqref{T1} for $\alpha=\alpha_0$, we use \eqref{rho:est1} to estimate the $L^{\alpha_0}$-norm in \eqref{linfest5}, and obtain 
\begin{align*}
\|u\|_{L^{\infty}(U'\times(\varepsilon ,T))}
&\le  C \varep^{-2\omega}(1+T^{3\omega})\max\Big\{ \Big(\int_0^T\int_{B_R}|u|^{\alpha_0}dxdt\Big)^{\frac{\mu}{\alpha_0}}, \Big(\int_0^T\int_{B_R}|u|^{\alpha_0}dxdt\Big)^{\frac{\nu}{\alpha_0}} \Big\}\\
&\le  C \varep^{-2\omega}(1+T^{3\omega})\max\Big\{ \Big(\int_0^T\mathcal{U}_{\alpha_0}dt\Big)^{\frac{\mu}{\alpha_0}}, \Big(\int_0^T\mathcal{U}_{\alpha_0}dt\Big)^{\frac{\nu}{\alpha_0}} \Big\},
\end{align*}
This proves \eqref{linfest4}.
%Note that we used $\theta_{\alpha_0} =\theta_*$ and $1+\midx_2 = \frac  {2-a}{(1-a)(1-\theta_{\alpha_0})} $.

In case $T$ satisfies \eqref{TesT} for $\alpha=\alpha_0$, we use estimate \eqref{rho:est2} in \eqref{linfest5} and note that $\nu\ge\mu$. Then we obtain \eqref{linfest6}.
\end{proof}

\myclearpage
\section{Global $L^\infty$-estimates}\label{Lglobal}

In this section, we estimate the $L^\infty$-norm on $U$ of the solution $u(x,t)$. We perform Moser's iteration on the entire domain, and take into account the effect of the Robin boundary condition. Thanks to the contribution of the boundary terms, calculations need to be much more meticulous.

%%================================
\begin{lemma}\label{lem58} 
%Assume $\alpha>2-\delta$ and $\alpha >n\mu_0$. 
Assume $\alpha> \max\{2-\delta, n\mu_0\}$.
Let $\theta$, $\mu_1$, $\mu_2$, $D_{3,\alpha}$, $D_{4,\alpha}$ be defined as in Lemma \ref{newtrace}.
 If $T>T_2>T_1\ge 0$ then 
\beq\label{Sest2}
\begin{aligned}
&\sup_{t\in[T_2,T]} \int_U u^{\alpha}dx +\int_{T_2}^T\int_U |\nabla u|^{2-a}u^{\alpha+\delta-2}dx dt\le  \mathcal M\tilde{\mathcal{S}}
\end{aligned}
\eeq
 with 
\beq\label{Stildef}
\tilde{\mathcal{S}}=\int_{T_1}^T \int_U u^{\alpha+\midx_1}dx dt+\Big(\int_{T_1}^T \int_U u^{\alpha+\midx_1}dx dt\Big)^\frac{\alpha+\delta-2}{\alpha+\midx_1},
\eeq
\beq\label{Calpha1}
\mathcal M=\mathcal M(\alpha,\varphi,T_1,T_2,T)
\eqdef c_{11}\alpha^2(E_{1}+E_{2}+E_{3}+E_{4}+E_{5}),
\eeq 
where constant $c_{11}>0$ is independent of $\alpha$, $T_1$, $T_2$, $T$, while 
%\beq\label{Ecoeffs}
\begin{align*}
E_{1}&=(T_2-T_1)^{-1}(|U|T)^{\frac{\midx_1}{\alpha+\midx_1}},\quad
E_{2}=(|U|T)^{\frac{\midx_1-\delta+2}{\alpha+\midx_1}},\quad
E_{3}=  (|U|T)^{\frac{\midx_1}{\alpha+\midx_1}}\|\varphi^-\|_{L^\infty(\Gamma\times (0,T))},\\
E_{4}&=  D_{3,\alpha} |U|^{\frac{\midx_1(\alpha+\mu_0)}{\alpha(\alpha+\midx_1)}}T^{\frac{\midx_1-\mu_0}{\alpha+\midx_1}} \|\varphi^-\|_{L^\infty(\Gamma\times (0,T))}^{\frac{2-a}{1-a}},\quad
E_{5}=  D_{4,\alpha}(d_3/4)^{ -\midx_2 }|U|^{\frac{\midx_1}{\alpha}}\|\varphi^-\|_{L^\infty(\Gamma\times (0,T))}^{\midx_4},
\end{align*}
%\eeq
with   $d_3>0$  defined in \eqref{i:ineq2}, and
\beq\label{muf} 
\midx_4=\midx_{4,\alpha}\eqdef \midx_2+1
=\frac{2-a}{(1-a)(1-\theta)}.\eeq
\end{lemma}
%===========================================%
 %\textcolor{red}{It should be $\varphi$ instead of $\psi$ in the third equality}
\begin{proof}
Let $\xi(t)$ be a smooth cut-off function with $\xi(t)\in[0,1]$,  $\xi(t)=0$ for $0\le t\le T_1$, $\xi(t)=1$ for $T_2<t<T$, and
\beq\label{xit}
|\xi_t|\le  C/(T_2-T_1),
\eeq
for some $C>0$ independent of $T_1,T_2$.

Multiplying the PDE in \eqref{rho:eq} by  $u^{\alpha+\delta-1}\xi^2$, integrating over $U$, and using integration by parts yield 
\begin{align*}
\lambda\int_U u^{\alpha-1}\xi^2\frac{\partial u}{\partial t}dx&=\int_U u^{\alpha+\delta-1}\xi^2\frac{\partial (u^{\lambda})}{\partial t}=\int_U \nabla \cdot (K (|\nabla u|)\nabla u  )u^{\alpha+\delta-1}\xi^2\\
&=-(\alpha-\lambda)\int_U K (|\nabla u|)|\nabla u|^2u^{\alpha+\delta-2}\xi^2dx -\int_\Gamma \varphi u^\alpha \xi^2 d\sigma .
\end{align*}
Using relation \eqref{i:ineq2} we obtain
\begin{align*}
\lambda\int_U u^{\alpha-1}\xi^2\frac{\partial u}{\partial t}dx\le -d_3(\alpha-\lambda)\int_U |\nabla u|^{2-a}u^{\alpha+\delta-2}\xi^2dx+d_3(\alpha-\lambda)\int_Uu^{\alpha+\delta-2}\xi^2dx+J(t),
\end{align*}
where %$d_3$ is a positive constant depending on $\vec\alpha$ and $\vec a$ and
\beqs
J(t)=\|\varphi^-(t)\|_{L^\infty(\Gamma)}\int_\Gamma  u^\alpha(x,t) \xi^2(t) d\sigma.
\eeqs
Since $\alpha-\lambda\ge 1$, we have
\begin{align*}
\lambda\int_U u^{\alpha-1}\xi^2\frac{\partial u}{\partial t}dx\le -d_3\int_U |\nabla u|^{2-a}u^{\alpha+\delta-2}\xi^2dx+d_3\alpha\int_Uu^{\alpha+\delta-2}\xi^2dx+J(t),
\end{align*}
Using the product rule on the left-hand side of the inequality we have
\beq\label{alphases}
\frac\lambda\alpha \ddt \int_U u^\alpha\xi^2 dx
+d_3\int_U |\nabla u|^{2-a}u^{\alpha+\delta-2}\xi^2dx
\le
\frac{2\lambda}\alpha\int_U u^\alpha\xi\xi_t dx
+d_3\alpha\int_Uu^{\alpha+\delta-2}\xi^2dx+J(t).
\eeq

Integrating from $0$ to $t$, and taking supremum for $t\in [0,T]$ give
\begin{multline}\label{phase2}
\frac\lambda\alpha \sup_{[0,T]} \int_U u^\alpha(x,t)\xi^2(t) dx
+d_3\int_0^T\int_U |\nabla u|^{2-a}u^{\alpha+\delta-2}\xi^2dxdt
\le
\frac{4\lambda}\alpha\int_0^T\int_U u^\alpha\xi|\xi_t| dxdt\\
+2 d_3\alpha\int_0^T \int_U u ^{\alpha+\delta-2}\xi^2dxdt+2\int_0^T J(t)dt.
\end{multline}

Using the trace inequality \eqref{trace-thm1} and noting that $\xi(t)$ is independent of $x$, we can estimate $J$ by
\begin{align*}
 J(t)
&\le   
\|\varphi^-(t)\|_{L^\infty(\Gamma)}  \Big\{ 2\varepsilon \int_U |u|^{\alpha+\delta-2}|\nabla u|^{2-a} \xi^2dx 
+  c_* \int_U u^\alpha \xi^2dx
+  D_{3,\alpha}\varepsilon^{-\frac  1{1-a}} \Big (\int_U u^\alpha \xi^2dx\Big)^\frac{\alpha+\mu_0}\alpha \\
 &\quad +  D_{4,\alpha}\varepsilon^{ - \midx_2 }
\Big(\int_U u^\alpha\xi^2 dx\Big)^\frac{\alpha+\midx_1}{\alpha}\Big\},
\end{align*}
where $\midx_2$ is defined by \eqref{muteen}. 
Hence,
\begin{align*}
2\int_0^T J(t) dt
&\le   
\|\varphi^-\|_{L^\infty(\Gamma\times (0,T))} \Big\{  4\varepsilon \int_0^T\int_U |u|^{\alpha+\delta-2}|\nabla u|^{2-a} \xi^2dx d\tau
+  C \int_0^T\int_U u^\alpha \xi^2dxd\tau\\
 &\quad + 2D_{3,\alpha}\varepsilon^{-\frac  1{1-a}} \int_0^T \Big(\int_U u^\alpha \xi^2dx\Big)^\frac{\alpha+\mu_0}\alpha d\tau
+ 2 D_{4,\alpha}\varepsilon^{ - \midx_2 }
\int_0^T\Big(\int_U u^\alpha\xi^2 dx\Big)^\frac{\alpha+\midx_1}{\alpha}d\tau\Big\}.
\end{align*}
Next, applying H\"older's inequality to the last three integrals 
%either in $xt$, or, in $x$ and then in $t$
yields
\beq\label{JintT}
2\int_0^T J(t) dt
\le   
   4\varepsilon \|\varphi^-\|_{L^\infty(\Gamma\times (0,T))} \int_0^T\int_U |u|^{\alpha+\delta-2}|\nabla u|^{2-a} \xi^2dx d\tau
+  C\|\varphi^-\|_{L^\infty(\Gamma\times (0,T))} J_0,
\eeq
where
\beqs
J_0=
 (|U|T)^{\frac{\midx_1}{\alpha+\midx_1}}Y^\frac{\alpha}{\alpha+\midx_1}
+ D_{3,\alpha}\varepsilon^{-\frac  1{1-a}} |U|^{\frac{\midx_1(\alpha+\mu_0)}{\alpha(\alpha+\midx_1)}}T^{\frac{\midx_1-\mu_0}{\alpha+\midx_1}}Y^\frac{\alpha+\mu_0}{\alpha+\midx_1} 
+  D_{4,\alpha}\varepsilon^{ - \midx_2 }|U|^{\frac{\midx_1}{\alpha}} Y
\eeqs
with $Y= \int_0^T \int_U u^{\alpha+\midx_1} \xi^2dx d\tau$.
Combining \eqref{phase2} and \eqref{JintT} with properties of $\xi(t)$ gives
\begin{multline}\label{supgradzr}
\frac{\lambda}{\alpha} \sup_{[0,T]} \int_U u^{\alpha}(x,t)\xi^2(t) dx
+(d_3-4\varepsilon\|\varphi^-\|_{L^\infty(U\times (0,T))} )\int_0^T\int_U |\nabla u|^{2-a}u^{\alpha+\delta-2}\xi^2dxd\tau\\
\le \frac{C}{\alpha(T_2-T_1)}\int_0^T\int_U u^{\alpha}\xi dxd\tau
 + 2d_3\alpha \int_0^T \int_U u^{\alpha+\delta-2} \xi^2 dx d\tau+ C\|\varphi^-\|_{L^\infty(\Gamma\times (0,T))}J_0\\
\le \frac{C}{\alpha(T_2-T_1)} (|U|T)^{\frac{\midx_1}{\alpha+\midx_1}}Y^\frac{\alpha}{\alpha+\midx_1}
+ 2d_3\alpha(|U|T)^{\frac{\midx_1-\delta+2}{\alpha+\midx_1}}Y^\frac{\alpha+\delta-2}{\alpha+\midx_1}
+C\|\varphi^-\|_{L^\infty(\Gamma\times (0,T))} J_0.
\end{multline}
Choosing $\varepsilon=\frac{d_3}{8\|\varphi^-\|_{L^\infty(U\times(0,T))}}$, and using properties of $\xi(t)$, we have
\begin{align*}
&\frac{\lambda}{\alpha} \sup_{t\in[T_2,T]} \int_U u^{\alpha}(x,t)dx 
+\frac {d_3}{2}\int_{T_2}^T\int_U |\nabla u|^{2-a}u^{\alpha+\delta-2}dxd\tau\\
&\le
C ( E_{1}/\alpha+E_{3} )Y^\frac{\alpha}{\alpha+\midx_1} +  C\alpha E_{2}Y^\frac{\alpha+\delta-2}{\alpha+\midx_1} 
+C E_{4}Y^\frac{\alpha+\mu_0}{\alpha+\midx_1}  
+C E_{5}Y.
\end{align*}
Note from the choice of the cut-off function $\xi(t)$ that $Y\le \tilde{Y} \eqdef\int_{T_1 }^T\int_U u^{\alpha+\midx_1}dxd\tau$.
Then
\beq\label{ti7}
 \sup_{t\in[T_2,T]} \int_U u^{\alpha}(x,t)dx \le C\alpha J_1
\quad\text{and}\quad
 \int_{T_2}^T\int_U |\nabla u|^{2-a}u^{\alpha+\delta-2}dxd\tau
\le C J_1,
\eeq
where
$ J_1 
= ( E_{1}/\alpha+E_{3} )\tilde{Y}^\frac{\alpha}{\alpha+\midx_1} +  \alpha E_{2}\tilde{Y}^\frac{\alpha+\delta-2}{\alpha+\midx_1}
+E_{4}\tilde{Y}^\frac{\alpha+\mu_0}{\alpha+\midx_1}  +  E_{5}\tilde{Y}.
$

Comparing the powers of $\tilde{Y}$ in $J_1$'s formula and applying \eqref{ee4}, we have
\beqs
J_1 
\le 3(E_{1}/\alpha+\alpha E_{2}+E_{3}+E_{4}+E_{5})\tilde {\mathcal S}\le 3\alpha (E_{1}+E_{2}+E_{3}+E_{4}+E_{5})\tilde {\mathcal S}.
\eeqs
Hence, 
%\Enote{$\lambda <1$, so $\frac{\alpha}{\lambda}>\alpha$. $\frac 1\lambda \to C$??}
\beqs
 \sup_{t\in[T_2,T]} \int_U u^{\alpha}(x,t)dx +  \int_{T_2}^T\int_U |\nabla u|^{2-a}u^{\alpha+\delta-2}dxd\tau
\le C\alpha J_1,
\eeqs
and we obtain \eqref{Sest2}.
\end{proof}
%================================================================%

%================================================================%
\begin{proposition}\label{GLk}
Assume $\alpha> \max\{2-\delta, n\mu_0\}$.
If $T>T_2>T_1>0$  then
%\Enote{erase this: and $\alpha>0$} 
\beq\label{bfinest2}
\| u\|_{L^{\kappa\alpha}(U\times (T_2,T))}
\le \tilde A_\alpha^\frac 1{\alpha}\Big( \| u\|_{L^{\alpha+\midx_1}(U\times(T_1,T))}^{\tilde r}+\| u\|_{L^{\alpha+\midx_1}(U\times(T_1,T))}^{\tilde s}\Big)^\frac 1{\alpha},
\eeq
where $\kappa$ is defined by \eqref{kappadef}, $\midx_1$ is defined by \eqref{muteen}, 
\beq\label{tilrsdef}
\tilde r=\tilde r(\alpha)\eqdef \frac{\alpha(\alpha+\delta-2)}{\alpha+\midx_1},\quad \tilde s= \tilde s(\alpha) \eqdef \frac{\alpha(\alpha+\midx_1)}{\alpha+\delta-a},
\eeq
\beq\label{Atildef}
\tilde A_{\alpha}=  c_{12}\alpha^{2-a} \big[(T|U|)^\frac{(\midx_1+a-\delta)\alpha}{(\alpha+\midx_1)(\alpha+\delta-a)}+\mathcal M^\frac{\alpha}{\alpha+\midx_1}+\mathcal M^{\frac{\alpha}{\alpha+\delta-a}}\big],\eeq
with $\mathcal M$ defined by \eqref{Calpha1}, and $c_{12}\ge 1$ independent of $\alpha$, $T_1$, $T_2$, $T$.
\end{proposition}
%================================================================%
\begin{proof} %The proof is similar to Theorem~\ref{preMoser}.  . . . . . 
We use Sobolev inequality  \eqref{parabwki} in Lemma \ref{Parabwk}:
\beq\label{beginest2}
\begin{aligned}
J\eqdef \Big(\int_{T_2}^T \int_U |u|^{\kappa\alpha}dx dt \Big)^{\frac 1{\kappa\alpha}}
&\le \hat c^\frac{1}{\kappa\alpha} \Big\{I^{\tilde \theta}\cdot \sup_{t\in[T_2,T]}\Big( \int_U|u|^{\alpha}dx \Big)^{1-\tilde \theta}\Big\}^\frac 1{\alpha},
\end{aligned}
\eeq
where $\hat c=c_5\alpha^{2-a}$, the numbers $\tilde \theta$ and $\kappa$ are defined in \eqref{kappadef} and 
\beqs
\begin{aligned}
I=\Big[  \int_{T_2}^T\int_U |u|^{\alpha+\delta-a}dxdt+\int_{T_2}^T\int_U |u|^{\alpha+\delta-2}|\nabla u|^{2-a}dxdt \Big]^{\frac {\alpha}{\alpha+\delta-a}} 
\end{aligned}
\eeqs
Applying inequality \eqref{ee3}, we find that   
\beq\label{Iest2}
I\le C_5 \Big(\int_{T_2}^T\int_U |u|^{\alpha+\delta-a}dxdt\Big)^{\frac {\alpha}{\alpha+\delta-a}}+C_5\Big(\int_{T_2}^T\int_U |u|^{\alpha+\delta-2}|\nabla u|^{2-a}dxdt \Big)^{\frac {\alpha}{\alpha+\delta-a}},
\eeq
where $C_5=2^{\frac{\alpha}{\alpha+\delta-a}-1}=2^\frac{a-\delta}{\alpha+\delta-a}$.
Applying H\"older's inequality to the first integral on the right-hand side of \eqref{Iest2} with conjugate exponents $\frac{\alpha+\midx_1}{\alpha+\delta-a}$ and $\frac{\alpha+\midx_1}{\midx_1-\delta+a}$, we get 
\begin{multline}\label{Iest3}
I\le C_5 (T|U|)^\frac{(\midx_1+a-\delta)\alpha}{(\alpha+\midx_1)(\alpha+\delta-a)}\Big(\int_{T_2}^T\int_U |u|^{\alpha+\midx_1}dx dt\Big)^\frac\alpha{\alpha+\midx_1}\\
+C_5\Big(\int_{T_2}^T\int_U |u|^{\alpha+\delta-2}|\nabla u|^{2-a}dx dt \Big)^{\frac {\alpha}{\alpha+\delta-a}}.
\end{multline}
Next, we use \eqref{Sest2} to estimate right-hand side of \eqref{Iest3}.
Hence combining $\eqref{beginest2}$ and $\eqref{Iest3}$ yields 
\begin{align*}
J&\le  \hat c^\frac{1}{\kappa\alpha} \Big\{ \Big [ C_5 (T|U|)^\frac{(\midx_1+a-\delta)\alpha}{(\alpha+\midx_1)(\alpha+\delta-a)}\tilde{\mathcal S}^\frac{\alpha}{\alpha+\midx_1}+C_5(\mathcal M\tilde{\mathcal S})^{\frac{\alpha}{\alpha+\delta-a}}\Big]^{\tilde \theta}(\mathcal M\tilde{\mathcal S} )^{1-\tilde \theta} \Big\}^{\frac 1{\alpha}}\\
&\le  \hat c^\frac{1}{\kappa\alpha} \Big\{  \Big [ C_5 (T|U|)^\frac{(\midx_1+a-\delta)\alpha}{(\alpha+\midx_1)(\alpha+\delta-a)}\tilde{\mathcal S}^\frac{\alpha}{\alpha+\midx_1}+C_5(\mathcal M\tilde{\mathcal S})^{\frac{\alpha}{\alpha+\delta-a}}+{\mathcal C}_\alpha\tilde{\mathcal S} \Big]^{\tilde \theta+1-\tilde \theta} \Big\}^{\frac 1{\alpha}}\\
&= \hat c^\frac{1}{\kappa\alpha} \Big\{   C_5 (T|U|)^\frac{(\midx_1+a-\delta)\alpha}{(\alpha+\midx_1)(\alpha+\delta-a)}\tilde{\mathcal S}^\frac{\alpha}{\alpha+\midx_1}+\mathcal M\tilde{\mathcal S}+C_5\mathcal M^{\frac{\alpha}{\alpha+\delta-a}}\tilde{\mathcal S}^{\frac{\alpha}{\alpha+\delta-a}} \Big\}^{\frac 1{\alpha}}.
\end{align*}
Since $\frac{\alpha}{\alpha+\midx_1}<1<\frac{\alpha}{\alpha+\delta-a}$, we use \eqref{ee4} to estimate
$\tilde{\mathcal S}\le \tilde{\mathcal S}^\frac{\alpha}{\alpha+\midx_1}+ \tilde{\mathcal S}^\frac{\alpha}{\alpha+\delta-a}.$
Thus, we have
\begin{align*}
J&\le \hat c^\frac{1}{\kappa\alpha} \Big\{  [ C_5 (T|U|)^\frac{(\midx_1+a-\delta)\alpha}{(\alpha+\midx_1)(\alpha+\delta-a)}+\mathcal M^\frac{\alpha}{\alpha+\midx_1}]\tilde{\mathcal S}^\frac{\alpha}{\alpha+\midx_1}+[1+C_5]\mathcal M^{\frac{\alpha}{\alpha+\delta-a}}\tilde{\mathcal S}^{\frac{\alpha}{\alpha+\delta-a}} \Big\}^{\frac 1{\alpha}}.
\end{align*}
Denote
$J_1=\big(\int_{T_1}^T \int_{U} |u|^{\alpha+\midx_1}dx dt \big)^{\frac 1{\alpha+\midx_1}}$.
Then by definition \eqref{Stildef} of $\tilde{\mathcal{S}}$, and applying \eqref{ee2}, resp. \eqref{ee3}, we find 
\beqs
\tilde{\mathcal S}^\frac{\alpha}{\alpha+\midx_1}
=(J_1^{\alpha+\midx_1}+J_1^{\alpha+\delta-2})^\frac{\alpha}{\alpha+\midx_1}\le J_1^\alpha+J_1^\frac{\alpha(\alpha+\delta-2)}{\alpha+\midx_1},
\eeqs
resp.,
\beqs
\tilde{\mathcal S}^{\frac{\alpha}{\alpha+\delta-a}} 
=(J_1^{\alpha+\midx_1}+J_1^{\alpha+\delta-2})^{\frac{\alpha}{\alpha+
\delta-a}}
\le C_5(J_1^\frac{\alpha(\alpha+\midx_1)}{\alpha+\delta-a}+J_1^\frac{(\alpha+\delta-2)\alpha}{\alpha+\delta-a}).
\eeqs
Therefore,
\beq\label{Jineq}
J\le \hat c^\frac{1}{\kappa\alpha} \Big\{ \tilde M_1 J_1^\alpha+\tilde M_1J_1^\frac{\alpha(\alpha+\delta-2)}{\alpha+\midx_1}+\tilde M_2 J_1^\frac{\alpha(\alpha+\midx_1)}{\alpha+\delta-a}+\tilde M_2 J_1^\frac{(\alpha+\delta-2)\alpha}{\alpha+\delta-a}\Big\}^{\frac 1{\alpha}},
\eeq
where
\beqs 
\tilde M_1= C_5  (T|U|)^\frac{(\midx_1+a-\delta)\alpha}{(\alpha+\midx_1)(\alpha+\delta-a)}+\mathcal M^\frac{\alpha}{\alpha+\midx_1},
 \quad 
 \tilde M_2=C_5(1+C_5)\mathcal M^{\frac{\alpha}{\alpha+\delta-a}}.
\eeqs
Since $\tilde r<\alpha,\frac{(\alpha+\delta-2)\alpha}{\alpha+\delta-a}<\tilde s$, we apply \eqref{ee4} to estimate the first and last summands on the right-hand side of \eqref{Jineq} by
\beqs
 J_1^\alpha,
 J_1^\frac{(\alpha+\delta-2)\alpha}{\alpha+\delta-a}
 \le 
 J_1^\frac{\alpha(\alpha+\delta-2)}{\alpha+\midx_1}
 +
 J_1^\frac{\alpha(\alpha+\midx_1)}{\alpha+\delta-a}.
 \eeqs
Then it follows
\beq \label{ti5}
J\le  \hat c^\frac{1}{\kappa\alpha}\Big\{  3(\tilde M_1+\tilde M_2)(J_1^{\tilde r}+J_1^{\tilde s})\Big\}^{\frac 1{\alpha}}.
\eeq 
Since $\alpha>2-\delta>2(a-\delta)$, we have $C_5\le 2$, hence $\tilde M_1+\tilde M_2\le 9\tilde{\mathcal M} $. Note also that $\hat c^\frac1\kappa\le \hat c$,
then we obtain \eqref{bfinest2} from \eqref{ti5}.
\end{proof}
%======================================%

%=========================================================%

Now to perform the iteration we need to start with an initial exponent $\kappa(\alpha_0)\alpha_0$ such that 
$\kappa(\alpha_0)\alpha_0 > \alpha_0 + \midx_{1,\alpha_0},$ that is, $\kappa(\alpha_0)\alpha_0 /( \alpha_0 + \midx_{1,\alpha_0})>1$.
We define
\beq\label{kapbar}
\bar \kappa(\alpha)=\frac{\kappa(\alpha)\alpha}{\alpha+\midx_{1,\alpha}}.
\eeq

The following properties are useful in later iterations.

%=========================================================%
\begin{lemma}\label{barkLem}
For $\alpha \in ((2-a)\alpha_*/(1-a),\infty)$, the functions $\alpha\to\mu_{1,\alpha}$ in \eqref{muteen} and $\alpha\to\mu_{4,\alpha}$ in \eqref{muf} are decreasing, while the function 
$\alpha \to \bar \kappa(\alpha)$ in \eqref{kapbar} is increasing. 
\end{lemma}
%=============================%
\begin{proof}
First of all, we note that $\kappa(\alpha)$ defined in \eqref{kappadef} is increasing, while $\theta_\alpha$ defined in \eqref{theta}  is decreasing.
Since $\midx_1$ defined in terms of $\theta$ in \eqref{muteen}  is increasing in $\theta$, and $\theta=\theta_\alpha$  is decreasing in $\alpha$, then, as a composition, $\midx_{1,\alpha}$ is decreasing in $\alpha$. Similar argument applies to $\mu_{4,\alpha}$.

Next, if $\alpha'>\alpha$ we have $\kappa(\alpha')\ge \kappa(\alpha)$ and $\mu_{1,\alpha'}\le \mu_{1,\alpha}$, hence
\beqs
\bar\kappa(\alpha')=\frac{\kappa(\alpha')\alpha'}{\alpha'+\midx_{1,\alpha'}}
\ge \frac{\kappa(\alpha)\alpha'}{\alpha'+\midx_{1,\alpha'}}
\ge \frac{\kappa(\alpha)\alpha'}{\alpha'+\midx_{1,\alpha}}
\ge \frac{\kappa(\alpha)\alpha}{\alpha+\midx_{1,\alpha}}=\bar \kappa(\alpha).
\eeqs
Therefore, $\bar\kappa(\alpha)$ is increasing in $\alpha$.
% Note that $$\midx_1(\alpha)=\frac{\mu_0(1+\theta(1-a))}{1-\theta}=\frac{\mu_0(1+\frac{\alpha_*}{\alpha-\alpha_*})}{1-\frac{\alpha_*}{(1-a)(\alpha-\alpha_*)}}=\frac{\frac{\alpha\mu_0}{\alpha-\alpha_*}}{\frac{(1-a)(\alpha-\alpha_*)-\alpha_*}{(1-a)(\alpha-\alpha_*)}}=\frac{\alpha\mu_0(1-a)}{(1-a)(\alpha-\alpha_*)-\alpha_*}$$
%Then $$\midx_1'(\alpha)=\frac{\mu_0(1-a)((1-a)(\alpha-\alpha_*)-\alpha_*)-\alpha\mu_0(1-a)^2}{[(1-a)(\alpha-\alpha_*)-\alpha_*]^2}<0$$
\end{proof}
%=========================================================%

We construct two sequences of exponents in order to implement Moser's iteration. (Regarding the notation, the numbers $\alpha_j$'s below are newly constructed and are not the exponents in \eqref{eq2}.)

%===========================================================================%

\begin{lemma}[Construction of $\alpha_j$'s and $\beta_j$'s.] \label{al-be-sq}
Let
\beq\label{xstar}
x_*= \frac{2+\sqrt{(2-a)(2+\frac 1n)-1}}{1-a}.
\eeq
Assume $\alpha_0>(1+x_*)\alpha_*$,
let $\theta_*=\theta_{\alpha_0}$, $\mu_*=\midx_{1,\alpha_0}$, $\kappa_*=\kappa(\alpha_0)$, and $\bar \kappa_*=\bar \kappa(\alpha_0)$.
Define the sequence $(\beta_j)_{j=0}^\infty$ by
\beq\label{be-Eq}
\beta_0=\alpha_0+\mu_*,\quad    \beta_j=\bar\kappa_*^j\beta_0 \quad \text{for }j\ge 1.
\eeq
Then:

\begin{enumerate}
\item $\bar\kappa_*>1$.

\item There exists a strictly increasing  sequence $(\alpha_j)_{j=0}^\infty$ such that 
  \beq \label{al-Eq}
 \beta_j=\alpha_j+\midx_{1,\alpha_j} \quad \forall j\ge 0.
  \eeq 
\item For all $j\ge 0$,
\beq\label{uni2}
\theta_{\alpha_j}\le \theta_*,\quad \midx_{1,\alpha_j}\le \mu_*,\quad \kappa(\alpha_j)\ge \kappa_*,
\quad \bar\kappa(\alpha_j)\ge \bar \kappa_*.
\eeq

\item For all  $j\ge 0,$
\beq\label{uni3}
\alpha_j<\bar \kappa_*^j \beta_0,
\eeq
and there exists a number $\hat \kappa_*>1$  such that
\beq\label{alpower2}
\alpha_j\ge \hat \kappa_*^j \alpha_0\quad \forall j\ge 0.
\eeq
\end{enumerate}
\end{lemma}
%==============================================================================%
\begin{proof}
(i)
Using definitions of $\kappa$, $\midx_1$ and $\mu_{0}$ in \eqref{kappadef}, \eqref{muteen} and \eqref{muzero},  the inequality $\bar\kappa_*>1$ is rewritten explicitly as
\beqs
\alpha_0\Big(1+(a-\delta)(\frac1{\alpha_*}-\frac{1}{\alpha_0})\Big)>\alpha_0 + \frac{\frac{a-\delta}{1-a}+\theta_{\alpha_0}(a-\delta)}{1-\theta_{\alpha_0}},
\eeqs
which is equivalent to
\beqs
\frac{\alpha_0}{\alpha_*}>\frac{2-a}{(1-a)(1-\theta_{\alpha_0})}.
\eeqs
Using formula \eqref{theta} for $\theta_{\alpha_0}$, we convert this inequality to a quadratic inequality in $\alpha_0$ as
\beqs
(1-a)\alpha_0^2-2(2-a)\alpha_* \alpha_0+  (2-a)\alpha_*^2>0.
\eeqs  
Its positive solutions are
\beq\label{a1}
 \alpha_0 > (1 + \frac {1+\sqrt {2-a}}{1-a})\alpha_*.
 \eeq
This is satisfied by our choice of $\alpha_0$. Therefore $\bar\kappa_*>1.$

(ii) It follows from definition of $\beta_j$ that $(\beta_j)_{j=0}^\infty,$ is unique and strictly increasing.  
Consider the equation 
\beq\label{eq-a1}
\beta_j=x+\midx_1(x)\eqdef f(x),
\eeq
where, for $x>0$, 
\beqs
f(x)=x + \frac{\mu_0+ (a-\delta)\Theta(x)}{1-\Theta(x)}
=x-(a-\delta) + \frac{(2-a)(a-\delta)}{(1-a)(1-\Theta(x))},
\eeqs
with $\Theta(x)=\frac{\alpha_*}{(1-a)(x-\alpha_*)}$.
%Since  $\Theta'(x) = - \frac{(1-a)\Theta^2(x)}{\alpha_*} = - \frac{(1-a)(2-a)\Theta^2(x)}{n(a-\delta)}$,  
We have
$f'(x)=1-\frac{(2-a)^2\Theta^2(x)}{n(1-\Theta(x))^2}$.
%\begin{align*}
%f'(x)
%=1+\frac{(2-a)(a-\delta)\Theta'(x)}{(1-a)(1-\Theta(x))^2}
%=1-\frac{(2-a)^2\Theta^2(x)}{n(1-\Theta(x))^2}.
%\end{align*}
Then $f'(x)>0$ if 
\beqs
 1/\Theta(x) > 1+\frac{2-a}{\sqrt n}, \text{ that is, }
x > \alpha_*\Big\{\frac1{1-a}\Big(1+\frac{2-a}{\sqrt n}\Big)+1\Big\}.
\eeqs
Note that $\frac{2-a}{\sqrt{n}}\le \sqrt{2-a}$,  we already have from \eqref{a1} that 
\beqs
\alpha_0>\Big[\frac 1{1-a}(\frac{2-a}{\sqrt{n}}+1)+1\Big]\alpha_*.
\eeqs
Hence,  $f$ is strictly increasing on $[\alpha_0,\infty)$, $f(\alpha_0)=\beta_0$ and $f(\infty)=\infty$.
Since the sequence $(\beta_j)$ is strictly increasing, we have for any $j\ge 1$ that  $\beta_j>\beta_0$, and hence 
the number $\alpha_j =f^{-1}(\beta_j)$ solves \eqref{al-Eq}.
Clearly, the sequence $(\alpha_j)_{j=0}^\infty$ is also strictly increasing.

(iii) 
By the monotonicity of $\theta_\alpha$, $\mu_{1,\alpha}$, $\kappa(\alpha)$,  $\bar \kappa(\alpha)$ (Lemma \ref{barkLem} and its proof),
and the fact $\alpha_j\ge \alpha_0$,
 we have 
\beqs
\theta_{\alpha_j}\le \theta_{\alpha_0}=\theta_*,\quad 
\midx_{1,\alpha_j}\le \midx_{1,\alpha_0}=\mu_*,\quad 
\kappa(\alpha_j)\ge \kappa(\alpha_0)=\kappa_*,\quad 
\bar\kappa(\alpha_j)\ge \bar\kappa(\alpha_0)=\bar \kappa_* 
\eeqs 

(iv) 
 From \eqref{al-Eq}, we have $\alpha_j<\beta_j$ and hence inequality \eqref{uni3} follows.     

 By \eqref{al-Eq} and the fact $\mu_{1,\alpha}$ is decreasing in $\alpha$, we have for $j\ge 1$ that 
\beqs
\bar\kappa_*=\frac{\beta_{j}}{\beta_{j-1}}
<\frac{\alpha_{j} +\mu_{1,\alpha_0}}{\alpha_{j-1}}
< \frac{\alpha_j}{\alpha_{j-1}} +\frac{\mu_0}{\alpha_0}\frac{1+\frac{1}{\alpha_0/\alpha_*-1}}{1-\frac{1}{(1-a)(\alpha_0/\alpha_*-1)}}.
\eeqs

Set $\alpha_0=(1+\hat x)\alpha_*$, then $\hat x>x_*$. Thus,
\beqs
\bar\kappa_*
< \frac{\alpha_j}{\alpha_{j-1}}+\frac{\mu_0(1-a)}{\alpha_*[(1-a)\hat x-1]}
=\frac{\alpha_j}{\alpha_{j-1}}+\frac{2-a}{n[(1-a)\hat x-1]}.
\eeqs
Hence,
\beq\label{ti10}
\frac{\alpha_j}{\alpha_{j-1}}-1> \hat\varepsilon \eqdef\bar\kappa_*-1-\frac{2-a}{n[(1-a)\hat x-1]}.
\eeq
We have 
\beqs
\bar \kappa_*-1
=\frac{\kappa_*\alpha_0}{\alpha_0+\mu_*}-1=\frac{(1+(a-\delta)(1/\alpha_*-1/\alpha_0))\alpha_0}{\alpha_0+\mu_*}-1
=\frac{(a-\delta)\hat x-\mu_*}{(1+\hat x)\frac{(a-\delta)n}{2-a}+\mu_*}.
\eeqs
Note that  $\theta_{\alpha_0}=\frac{1}{(1-a)\hat{x}}$, hence,
$$\mu_*=\frac{\mu_0(1+\theta_{\alpha_0}(1-a))}{1-\theta_{\alpha_0}}=\frac{(a-\delta)(1+\frac{1}{\hat{x}})}{(1-a)(1-\frac{1}{(1-a)\hat{x}})}
=\frac{(a-\delta)(\hat{x}+1)}{(1-a)\hat{x}-1}.$$
Thus,
\beqs
\bar \kappa_*-1
=\frac{\frac{\hat x}{1+\hat x}-\frac1{(1-a)\hat x-1}}{\frac {n}{2-a}+\frac1{(1-a)\hat x-1}}.
\eeqs
We aim at finding $\hat x$ such that
\beq\label{suffcond}
\hat\varepsilon = \frac{\frac{\hat x}{1+\hat x}-\frac1{(1-a)\hat x-1}}{\frac {n}{2-a}+\frac1{(1-a)\hat x-1}}- \frac{2-a}{n[(1-a)\hat x-1 ]}>0.
\eeq

If  inequality \eqref{suffcond} holds true then we choose $\hat \kappa_*=1+\hat\varepsilon>1$, and  by \eqref{ti10},
$\alpha_j/\alpha_{j-1}\ge 1+\hat\varepsilon=\hat \kappa_*$.
Thus, $\alpha_j\ge \hat \kappa_*\alpha_{j-1}$, and by induction, $\alpha_j\ge \hat \kappa_*^j\alpha_0$ for all $j\ge 0$.

It remains to verify  \eqref{suffcond}. For  $\hat x>\frac 1 {1-a}$, inequality \eqref{suffcond} is equivalent to
\beq\label{midstep}
(1-a)^2\hat x^3 -4(1-a) \hat x^2 +(1+2a -\frac{2-a}{n})\hat x+2 -\frac{2-a}{n}>0.
\eeq
% \beq
% \begin{aligned}
% &\Big(\frac{\hat x}{1+\hat x}-\frac1{(1-a)\hat x-1}\Big)\frac{n[(1-a)\hat x-1 ]}{2-a}   >\frac {n}{2-a}+\frac1{(1-a)\hat x-1}\\
% \Longleftrightarrow &\quad\hat x[(1-a)\hat x -1 ]^2 -\frac{2-a}{n}(\hat x+1)> 2(\hat x+1)[(1-a)\hat x-1]\\
% \Longleftrightarrow &\quad(1-a)^2\hat x^3 -4(1-a) \hat x^2 +(1+2a -\frac{2-a}{n})\hat x+2 -\frac{2-a}{n}>0.
% \end{aligned}
% \eeq
Since $2-\frac{2-a}{n} >0$, a sufficient condition for  \eqref{midstep} is
\beqs
(1-a)^2\hat x^2-4(1-a) \hat x +(1+2a -\frac{2-a}{n})>0.
\eeqs
Solving this inequality gives  
\beqs
\hat x>\frac {2+\sqrt{3-2a +\frac {2-a}{n} }}{1-a}=x_*,
\eeqs
which is satisfied by the choice of $\hat x$. The proof is complete.
\end{proof}

%=========================================================%

The final preparation for Moser's iteration is to estimate $\tilde A_\alpha $ in \eqref{bfinest2}.

%=========================================================%
\begin{lemma}\label{CAlem}
Let  $\alpha_0$ be a positive number such that $\alpha=\alpha_0$ satisfies \eqref{alpcond}.
Then one has for any $\alpha\ge \alpha_0$ that
\beq\label{tilAest}
\tilde A_\alpha \le \hat C \alpha^{\midx_5}(1+T)^{\midx_6}(1+\frac1{T_2-T_1})^2(1+\|\varphi^-\|_{L^\infty(\Gamma\times (0,T))})^{\midx_7},
\eeq
where $\midx_5,\midx_6,\midx_7>0$  and $\hat C>0$ depend on $\alpha_0$ but not on $\alpha$.
\end{lemma}
%=========================================================%
\begin{proof}
In this proof, $\hat C$ denotes a generic positive constant depending on $\alpha_0$, but not on $\alpha$.

In order to estimate $\tilde A_\alpha$, we estimate $\mathcal M$ from \eqref{Calpha1} first.
Note that  $0\le \midx_2\le \midx_{2,\alpha_0}$, hence $(d_3/4)^{-\mu_2}\le \hat C$.
Then from \eqref{Calpha1}, we find 
\begin{align*}
\mathcal M 
&\le  \hat C \alpha^2(1+\frac1{T_2-T_1})(1+\|\varphi^-\|_{L^\infty(\Gamma\times (0,T))})^{\midx_4}\Big\{ (|U|T)^{\frac{\midx_1}{\alpha+\midx_1}}+(|U|T)^{\frac{\midx_1-\delta+2}{\alpha+\midx_1}}\\
&\quad + D_{3,\alpha} |U|^{\frac{\midx_1(\alpha+\mu_0)}{\alpha(\alpha+\midx_1)}}T^{\frac{\midx_1-\mu_0}{\alpha+\midx_1}}+D_{4,\alpha} |U|^{\frac{\midx_1}{\alpha}}\Big\}.
\end{align*}
Above, we used the fact that $\midx_4>1$ is the maximum among the exponents of $\|\varphi^-\|_{L^\infty(\Gamma\times (0,T))}$. Next, using definitions of $D_{3,\alpha}$ and $D_{4,\alpha}$ in \eqref{newD1} and \eqref{newD2}, we have
\begin{align*}
\mathcal M 
&\le \hat C \alpha^2(1+\frac1{T_2-T_1})(1+\|\varphi^-\|_{L^\infty(\Gamma\times (0,T))})^{\midx_4}\Big\{ (|U|T)^{\frac{\midx_1}{\alpha+\midx_1}}+(|U|T)^{\frac{\midx_1-\delta+2}{\alpha+\midx_1}}\\
&\quad +   2^{\theta(\alpha+\delta-a)}c_*^\frac{(2-a)(1+\theta(1-a))}{1-a}  \alpha^\frac{2-a}{1-a}|U|^{\frac{(1-a)(\alpha+\mu_0)\theta}{\alpha}}
\cdot |U|^{\frac{\midx_1(\alpha+\mu_0)}{\alpha(\alpha+\midx_1)}}T^{\frac{\midx_1-\mu_0}{\alpha+\midx_1}}\\
&\quad +2^\frac{\theta(\alpha+\delta-a)}{1-\theta}
(c_*\alpha)^\frac{(2-a)(1+\theta(1-a))}{(1-a)(1-\theta)}|U|^{\frac{\midx_1}{\alpha}}\Big\},
\end{align*}
where constants are as in Lemma \ref{newtrace}. 
For convenience in calculations below, we keep using $c_*$ to denote $\max\{1,c_*\}$.
Also, using $\theta<1$ and $m\le \alpha$, we have
\begin{align*}
\mathcal M
&\le \hat C \alpha^2 (1+\frac1{T_2-T_1})
2^\frac{\theta(\alpha+\delta-a)}{1-\theta}
(c_*\alpha)^\frac{(2-a)(1+\theta(1-a))}{(1-a)(1-\theta)}(1+\|\varphi^-\|_{L^\infty(\Gamma\times (0,T))})^{\midx_4}\\
&\quad \cdot \Big\{ (|U|T)^{\frac{\midx_1}{\alpha+\midx_1}}+(|U|T)^{\frac{\midx_1-\delta+2}{\alpha+\midx_1}} + |U|^{\frac{(1-a)(\alpha+\mu_0)\theta}\alpha + \frac{\midx_1(\alpha+\mu_0)}{\alpha(\alpha+\midx_1)}} T^{\frac{\midx_1-\mu_0}{\alpha+\midx_1}} + |U|^{\frac{\midx_1}{\alpha}}\Big\}.
\end{align*}

Set $\theta_*=\theta_{\alpha_0}$ and  $\mu_*=\midx_{1,\alpha_0}$.
Since $\alpha\ge \alpha_0$, we have $\theta_\alpha\le \theta_*$.
Hence
\beqs
(c_*\alpha)^\frac{(2-a)(1+\theta(1-a))}{(1-a)(1-\theta)}
\le (c_*\alpha)^\frac{(2-a)^2}{(1-a)(1-\theta)}
\le (c_*\alpha)^\frac{(2-a)^2}{(1-a)(1-\theta_*)}
\le \hat C \alpha^{z_1},\text{ where }z_1=\frac{(2-a)^2}{(1-a)(1-\theta_*)}.
\eeqs

Next, we want  to bound $2^\frac{\theta (\alpha+\delta-a)}{1-\theta}$ by some number independent of $\alpha$.
From \eqref{theta},
\beqs
\theta (\alpha+\delta-a)=\frac{\alpha+\delta-a}{(1-a)(\alpha/\alpha_*-1)},
\eeqs
which, due to the fact that $\alpha_*>a-\delta$, is decreasing in $\alpha$.
Hence,
\beqs
\theta (\alpha+\delta-a)
\le z_2\eqdef \theta_* (\alpha_0+\delta-a)\text{ and }
2^\frac{\theta (\alpha+\delta-a)}{1-\theta}
\le 2^\frac{z_2}{1-\theta_*}.
\eeqs

For exponents of $|U|T$, we note that
\beqs
\frac{\midx_1}{\alpha+\midx_1}, \frac{\midx_1-\delta+2}{\alpha+\midx_1}\le 1.
\eeqs

For the remaining power of $T$,
\beqs
\frac{\midx_1-\mu_0}{\alpha+\midx_1}\le 1.
\eeqs

For the remaining powers of $|U|$,
\beqs
(1-a)\theta\frac{\alpha+\mu_0}{\alpha}\le \frac{\alpha+\mu_0}{\alpha}\le 2,\quad
\frac{\midx_1(\alpha+\mu_0)}{\alpha(\alpha+\midx_1)}
\le \frac{\midx_1}{\alpha}\le \mu_*,
\eeqs

Also, we have for the power of $\|\varphi^-\|_{L^\infty(\Gamma\times (0,T))}$:
\beqs \midx_4\le z_3\eqdef \mu_{4,\alpha_0}=\frac{2-a}{(1-a)(1-\theta_*)},
\eeqs 
due to the decrease of $\mu_4$ in $\alpha$, see Lemma \ref{barkLem}.

So we find
\beqs
\mathcal M 
\le \hat C \alpha^{2+z_1}(1+\frac1{T_2-T_1})2^\frac{ z_2}{1-\theta_*} (1+\|\varphi^-\|_{L^\infty(\Gamma\times (0,T))})^{z_3} (1+|U|)^{2+\mu_*}(1+T).
\eeqs

Summing up, we obtain
\beq\label{Cest}
\mathcal M \le \bar {\mathcal M} \eqdef \hat C \alpha^{2+z_1}(1+\frac1{T_2-T_1})(1+T)
(1+\|\varphi^-\|_{L^\infty(\Gamma\times (0,T))})^{z_3}.
\eeq

Since $\bar {\mathcal M}>1$ and $\frac{\alpha}{\alpha+\midx_1}< \frac{\alpha}{\alpha+\delta-a}\le 2$, using \eqref{Cest} in \eqref{Atildef} we have
\beq\label{ti9}
\tilde A_{\alpha}\le  c_{12}\alpha^{2-a} [(T|U|)^\frac{(\midx_1+a-\delta)\alpha}{(\alpha+\midx_1)(\alpha+\delta-a)}+2\bar {\mathcal M}^2].
 \eeq
Note that $a-\delta=(1-a)\mu_0$ and $\alpha+\delta-a\ge 2-a$, hence
\beq\label{ti100}
\frac{(\midx_1+a-\delta)\alpha}{(\alpha+\midx_1)(\alpha+\delta-a)}=\frac{\mu_0(2-a) \alpha}{(1-\theta)(\alpha+\midx_1)(\alpha+\delta-a)}
\le \frac{\mu_0(2-a) }{(1-\theta_*)(\alpha+\delta-a)}
\le \frac{\mu_0}{1-\theta_*}.
\eeq
Hence by \eqref{Cest}, \eqref{ti9} and \eqref{ti100},
\beq\label{Mtilest}
\tilde A_{\alpha}\le \hat C \alpha^{2-a+2(2+z_1)}(1+\frac1{T_2-T_1})^2 (1+T)^{\max\big\{2, \frac{\mu_0}{1-\theta_*} \big\}} (1+\|\varphi^-\|_{L^\infty(\Gamma\times (0,T))})^{2z_3}.
\eeq
Hence we obtain \eqref{tilAest} from \eqref{Mtilest} with $\midx_5=6-a+2z_1$, $\midx_6=\max\big\{2, \frac{\mu_0}{1-\theta_*} \big\}$, 
and $\midx_7=2z_3$.
\end{proof}
%===========================================================================%

Applying iteration process, we obtain:

%=========================================================%
\begin{theorem}\label{LinfU} 
Assume
\beq\label{finalal}
\alpha_0> \max\{2-\delta,(1+x_*)\alpha_*\}
\eeq
with $x_*$ defined by \eqref{xstar}. There are  $C,\tilde  \mu,\tilde \nu,\omega_1,\omega_2,\omega_3>0$ such that
if  $T>0$ and $\sigma\in (0,1)$ then
\begin{multline}\label{Li1}
\|u\|_{L^{\infty}(U\times(\sigma T,T))}\le C\Big(1+\frac1{\sigma T}\Big)^{\omega_1}(1+T)^{\omega_2}(1+\|\varphi^-\|_{L^\infty(\Gamma\times (0,T))})^{\omega_3}\\
\cdot \max\Big\{ \|u\|^{\tilde \mu}_{L^{\beta_0}(U\times(0,T))},\|u\|^{\tilde \nu}_{L^{\beta_0}(U\times(0,T))}\Big\}, 
\end{multline}
% \beq
% C_{T,\sigma,\varphi}=C(1+\frac1{\sigma T})^{\omega_1}(1+T)^{\omega_2}(1+\|\varphi^-\|_{L^\infty(\Gamma\times (0,T))})^{\omega_3},
% \eeq
where $\beta_0=\alpha_0+\midx_{1,\alpha_0}$.
\end{theorem}
%=========================================================%
\begin{proof}

Note that $\alpha_*\ge a-\delta$ and $n\mu_0=\alpha_*(2-a)/(1-a)$, then it is easy to check that $\alpha_0\ge 2(a-\delta)$ and $\alpha_0>n\mu_0$.

Let $\midx_5$, $\midx_6$ and $\midx_7$  be defined as in Lemma \ref{CAlem}, and $\alpha_j$, $\beta_j$, $\kappa_*$, $\bar \kappa_*$, $\hat\kappa_*$ be as in Lemma ~\ref{al-be-sq}.

For $j\ge 0$, let $t_j=\sigma T(1-\frac 1{2^j})$, $Q_j=U\times (t_j,T)$, and define
$Y_j=\| u\|_{L^{\beta_j}(Q_j)}.$

Applying \eqref{bfinest2} of Theorem \ref{GLk} with $\alpha=\alpha_j$, $T_2=t_{j+1}$ and $T_1=t_j$, we have 
\begin{align*}
\| u\|_{L^{\kappa(\alpha_j)\alpha_j}(U\times (t_{j+1},T))}
&\le \tilde A_{\alpha_j}^\frac 1{\alpha_j}\Big( \| u\|_{L^{\alpha_j+\midx_{1,\alpha_j}}(U\times(t_j,T))}^{\tilde r(\alpha_j)}+\| u\|_{L^{\alpha_j+\midx_{1,\alpha_j}}(U\times(t_j,T))}^{\tilde s(\alpha_j)}\Big)^\frac 1{\alpha_j}\\
&= \tilde A_{\alpha_j}^\frac 1{\alpha_j}\Big( \| u\|_{L^{\beta_j}(U\times(t_j,T))}^{\tilde r_j}+\| u\|_{L^{\beta_j}(U\times(t_j,T))}^{\tilde s_j}\Big)^\frac 1{\alpha_j},
\end{align*}
where
$\tilde r_j=\tilde r(\alpha_j)$ and $\tilde s_j=\tilde s(\alpha_j)$, see formula \eqref{tilrsdef}.

By part (iii) of Lemma \ref{al-be-sq}, $\kappa(\alpha_j)\alpha_j=\bar\kappa(\alpha_j) \beta_j\ge \bar\kappa_* \beta_j=\beta_{j+1}$, then by H\"older's inequality
\beqs
 Y_{j+1}= \| u\|_{L^{\beta_{j+1}}(Q_{j+1})}\le  |Q_{j+1}|^{\frac{1}{\beta_{j+1}}-\frac{1}{\kappa(\alpha_{j})\alpha_{j}}}  \| u\|_{L^{\kappa(\alpha_j)\alpha_j}(Q_{j+1})}.
\eeqs
Combining the above two inequalities give
\beq\label{YwithQ}
Y_{j+1}\le \widehat A_j^\frac{1}{\alpha_j}\big( Y_j^{\tilde{r}_j}+Y_j^{\tilde{s}_j}\big)^{\frac 1{\alpha_j}},\quad \text{where }
\widehat A_j=  |Q_{j+1}|^{\frac{\alpha_j}{\beta_{j+1}}-\frac{1}{\kappa(\alpha_j)}} \tilde A_{\alpha_j}.
\eeq

Now we estimate $\widehat{A}_j$. From \eqref{tilAest}, \eqref{YwithQ}, the fact that $ \alpha_j <\beta_{j+1}$, and \eqref{uni3}, we have 
\begin{align*}
\widehat A_j
&\le C  (1+|Q_{j+1}|)^\frac{\alpha_j}{\beta_{j+1}} \alpha_j^{\midx_5}(1+\frac{2^{j+1}}{\sigma T})^2(1+T)^{\midx_6}(1+\|\varphi^-\|_{L^\infty(\Gamma\times (0,T))})^{\midx_7}\\
&\le C(1+|Q_{0}|)(\bar \kappa_*^j\beta_0)^{\midx_5}4^j(1+\frac{1}{\sigma T})^2(1+T)^{\midx_6}(1+\|\varphi^-\|_{L^\infty(\Gamma\times (0,T))})^{\midx_7}
\le A_{T,\sigma,\varphi}^{j+1},
\end{align*}
where $$ A_{T,\sigma,\varphi}=\max\big\{ 4\bar \kappa_*^{\midx_5},C \beta_0^{\midx_5}(1+\frac{1}{\sigma T})^2(1+|U|T)(1+T)^{\midx_6}(1+\|\varphi^-\|_{L^\infty(\Gamma\times (0,T))})^{\midx_7} \big\}> 1.$$
Hence 
\beq\label{YY}
Y_{j+1}\le A_{T,\sigma,\varphi}^{\frac{j+1}{\alpha_j}}\big(Y_j^{\tilde{r}_j}+Y_j^{\tilde{s}_j}\big)^{\frac 1{\alpha_j}}.
\eeq

From \eqref{alpower2} we have
\beqs
\sum_{j=1}^\infty \frac {j+1} { \alpha_j}\le \frac 1{ \alpha_0}\sum_{j=1}^\infty \frac {j+1} {\hat \kappa_*^j}<\infty.
\eeqs

Note that
\beq\label{q1}
1\ge \frac{\tilde r_j}{\alpha_j}=\frac{\alpha_j+\delta-2}{\alpha_j+\midx_{1,\alpha_j}}\ge \frac{\alpha_j+\delta-2}{\alpha_j+\mu_*}\ge \frac{ \hat \kappa_*^j \alpha_0+\delta-2}{ \hat \kappa_*^j \alpha_0+\mu_*}\ge 1 - \frac{\mu_*+2-\delta }{\hat \kappa_*^j \alpha_0} , 
\eeq
\beq\label{q2}
1\le \frac{\tilde s_j}{\alpha_j}=\frac{\alpha_j+\midx_{1,\alpha_j}}{\alpha_j+\delta-a}
\le \frac{\alpha_j+\mu_*}{\alpha_j+\delta-a}\le \frac{\hat \kappa_*^j \alpha_0+\mu_*}{\hat \kappa_*^j \alpha_0+\delta-a}= 1+\frac{\mu_*+a-\delta }{\hat \kappa_*^j \alpha_0+\delta -a}.
\eeq
Then it is elementary, see \eqref{cv0} and \eqref{converg}, to show that the products 
\beq \label{mnutil}
\tilde \mu =  \Pi_{j=0}^\infty \frac{\alpha_j+\delta-2}{\alpha_j+\midx_{1,\alpha_j}}\quad\text{and}\quad \tilde \nu = \Pi_{j=0}^\infty \frac{\alpha_j+\midx_{1,\alpha_j}}{\alpha_j+\delta-a}
\eeq 
converge to  positive numbers.
% The proof for convergence is similar to calculations in \eqref{converg}.
By \eqref{YY} and Lemma \ref{Genn}, we obtain
\beq\label{limY}
\limsup_{j\to\infty} Y_{j}\le (2A_{T,\sigma,\varphi})^\omega\max\{Y_0^{\tilde\mu}, Y_0^{\tilde\nu}\},
\eeq
where $\omega=\mathcal G\sum_{j=1}^{\infty}\frac {j+1}{\alpha_j}$ with 
%$\mathcal G=\limsup_{j\to\infty}\max\{\prod_{k=\ell}^{\ell'} (\tilde s_k/\alpha_k): 1\le \ell\le \ell'\le j\}\in(0,\infty)$,
$\mathcal G=\prod_{k=1}^\infty (\tilde s_k/\alpha_k)\in(0,\infty)$.

%and 
%\beq\label{e1}
% \tilde \mu =  \Pi_{j=0}^\infty \frac{\alpha_j+\delta-2}{\alpha_j+\midx_{1,\alpha_j}},\quad
% \tilde \nu = \Pi_{j=0}^\infty \frac{\alpha_j+\midx_{1,\alpha_j}}{\alpha_j+\delta-a}.
%\eeq
Note that
\beqs
(2A_{T,\sigma,\varphi})^\omega\le C(1+\frac1{\sigma T})^{\omega_1}(1+T)^{\omega_2}(1+\|\varphi^-\|_{L^\infty(\Gamma\times (0,T))})^{\omega_3},
\eeqs
where $\omega_1=2\omega$, $\omega_2=(1+\midx_6)\omega$ and $\omega_3=\midx_7\omega$.
Then estimate \eqref{Li1} follows \eqref{limY}.
\end{proof}

%==================================================================%

\begin{remark}
{\rm (i)}  The exponents $\tilde \mu$ and $\tilde \nu$ in \eqref{Li1} are given by \eqref{mnutil} but can, in fact, be replaced by simpler and more explicit ones such as 
\beqs
\hat \mu=\Pi_{j=0}^\infty \frac{\alpha_0 \hat \kappa_*^j-2+\delta}{\alpha_0 \hat \kappa_*^j +\mu_*},\quad \hat \nu=\Pi_{j=0}^\infty \frac{\alpha_0\hat \kappa_*^j+\mu_*}{\alpha_0\hat \kappa_*^j+\delta-a}.
\eeqs
Indeed, it follows from estimates \eqref{q1} and \eqref{q2} that $\hat\mu\le\tilde \mu\le\tilde \nu\le \hat \nu$, and then applying \eqref{ee4} gives
\beqs
\max\{Y_0^{\tilde \mu},Y_0^{\tilde \nu}\}\le Y_0^{\hat  \mu}+Y_0^{\hat \nu}\le 2\max\{Y_0^{\hat \mu},Y_0^{\hat \nu}\}.
\eeqs

{\rm (ii)} Estimate \eqref{Li1} is a global version of the improvement \eqref{Mishaineq} on interior estimates. See also \cite{HK2} for a similar global result for degenerate equations with the use of De Giorgi's iteration instead.
\end{remark}

We now have global  $L^\infty$-estimates in terms of the initial and boundary data.

%==================================================================%
\begin{theorem}\label{LinfData} 
Assume $\alpha_0$ satisfies \eqref{finalal}.  
Let $\beta_0=\alpha_0 +\midx_{1,\alpha_0}$.  Then there are positive numbers $C,\omega_1,\omega_2,\omega_3,\tilde \mu, \tilde \nu$ such that:

\begin{enumerate}
\item If $T >0$ satisfies \eqref{T1} with $\alpha=\beta_0$, and $0<\varepsilon<\min\{1,T\}$, then
\begin{multline} \label{GlobU}
\|u\|_{L^{\infty}(U\times(\varepsilon,T))}\le  C \varepsilon^{-\omega_1}(1+T)^{\omega_2}(1+\|\varphi^-\|_{L^\infty(\Gamma\times(0,T))})^{\omega_3}\\
\cdot \max\Big\{ \Big(\int_0^T\mathcal{U}_{\beta_0}(t) dt\Big)^{\frac{\tilde \mu}{\beta_0}}, \Big(\int_0^T\mathcal{U}_{\beta_0}(t)dt\Big)^{\frac{\tilde \nu}{\beta_0}} \Big\},
\end{multline}
where $ \mathcal{U}_{\beta_0}(t)$ is defined by \eqref{phidef} with $\beta_0$ replacing $\alpha_0$.
%$$
% \mathcal{U}_{\beta_0}(t)= \Big\{  \Big( 1+\int_U u_0^{\beta_0}(x) dx\Big)^{-\frac{\midx_{1,\beta_0}}{\beta_0}} - C_3 \int_0^t (1+\norm{\varphi^-(\tau)}_{L^\infty(\Gamma)} ^\frac{2-a}{(1-a)(1-\theta_{\beta_0})}) d\tau  \Big\}^{-\frac{\beta_0}{\midx_{1,\beta_0}}}.
%$$
%with $C_3>0$ defined in Theorem \ref{est-sol} for $\alpha=\beta_0$.

\item If $T>0$ satisfies \eqref{TesT} with $\alpha=\beta_0$, then 
\beq \label{GlobU2}
\|u\|_{L^{\infty}(U\times(\varep,T))}\le C \varepsilon^{-\omega_1}(1+T)^{\omega_2+\tilde \nu/\beta_0}
\big(1+\| u_0\|_{L^{\beta_0}(U)}\big)^{\tilde \nu}
(1+\|\varphi^-\|_{L^\infty(\Gamma\times(0,T))})^{\omega_3}.
\eeq
\end{enumerate}
\end{theorem}
%==========================$
\begin{proof}
(i) Applying \eqref{Li1} to $\sigma T=\varepsilon>0$, we have
\begin{multline}\label{m1}
\|u\|_{L^{\infty}(U\times(\varep,T))}
\le C \varepsilon^{-\omega_1}(1+T)^{\omega_2}(1+\|\varphi^-\|_{L^\infty(\Gamma\times(0,T))})^{\omega_3}\\
\cdot \max\Big\{ \Big(\int_0^T\int_U|u|^{\beta_0}dxdt\Big)^{\frac{\tilde \mu}{\beta_0}}, \Big(\int_0^T\int_U|u|^{\beta_0}dxdt\Big)^{\frac{\tilde \nu}{\beta_0}} \Big\}.
\end{multline}
Using \eqref{rho:est1} we have 
%\beqs
% \int_U u^{\beta_0}(x,t) dx \le  \Big\{  \Big( 1+\int_U u_0^{\beta_0}(x) dx\Big)^{-\frac{\midx_1}{\beta_0}} - C_3 \int_0^t (1+\norm{\varphi^-(\tau)}_{L^\infty(\Gamma)} ^\frac{2-a}{(1-a)(1-\theta(\beta_0))}) d\tau  \Big\}^{-\frac{\beta_0}{\midx_1}}.
%\eeqs 
%That is we have
\beq\label{m2} \int_U|u(x,t)|^{\beta_0}dx \le \mathcal{U}_{\beta_0}(t).
\eeq
% So 
% \beqs
% \begin{aligned}
% \|u\|_{L^{\infty}(U\times(\varep,T))}
% &\le C N_{T,\varep,\varphi}\max\Big\{ \Big(\int_0^T\mathcal{U}_{\varphi,\beta_0}dt\Big)^{\frac{\tilde \mu}{\beta_0}}, \Big(\int_0^T\mathcal{U}_{\varphi,\beta_0}dt\Big)^{\frac{\tilde \nu}{\beta_0}} \Big\}.
% \end{aligned}
% \eeqs
Then  \eqref{GlobU} follows \eqref{m1} and \eqref{m2}.

(ii) If $T>0$ satisfies \eqref{TesT} with $\alpha=\beta_0$, then \eqref{rho:est2} gives
\beqs
 \int_U u^{\beta_0}(x,t) dx \le  2 \left(1+\int_U u_0^{\beta_0}(x) dx\right).
\eeqs 
Combining this with \eqref{m1} yields
\begin{multline*}
\|u\|_{L^{\infty}(U\times(\varep,T))}
\le C \varepsilon^{-\omega_1}(1+T)^{\omega_2}(1+\|\varphi^-\|_{L^\infty(\Gamma\times(0,T))})^{\omega_3}\\
\cdot\max\Big\{ \Big(2\int_0^T (1+\int_U u_0^{\beta_0}(x) dx)dt\Big)^{\frac{\tilde \mu}{\beta_0}}, \Big(2\int_0^T(1+\int_U u_0^{\beta_0}(x) dx)dt\Big)^{\frac{\tilde \nu}{\beta_0}} \Big\}.
\end{multline*}
Since $\tilde \nu>\tilde \mu$, then we obtain \eqref{GlobU2}.
\end{proof}

%==================================================================%

%==================================================================%

\appendix

%==================================================================%
\myclearpage
\section{Appendix} 

Let $(y_j)_{j=0}^\infty$ be a sequence of non-negative numbers. 
%We start with a simple result which is a small variation of the main idea in Moser's iteration \cite{moser1971pointwise}. 
%Its precise estimates will be used for later extension.

%==================================================================%
\begin{lemma}\label{lemb1}
Let $(\alpha_j)_{j=0}^\infty$ and $(\beta_j)_{j=0}^\infty$ be sequences of positive numbers with
\beqs
\bar \alpha\eqdef\sum_{j=0}^\infty \alpha_j<\infty \text{ and }\bar \beta\eqdef \prod_{j=0}^\infty\beta_j \text{ exists and belongs to }\in(0,\infty).
\eeqs
Suppose there is $A\ge 1$ such that
\beq\label{A2}
y_{j+1}\le A^{\alpha_j} y_j^{\beta_j}\quad \forall j\ge 0.
\eeq
Then $(y_j)_{j=0}^\infty$ is a bounded sequence. More specifically, for all $j\ge 1$,
\beq\label{AS}
y_{j}\le A^{B_{j}\sum_{i=0}^{j-1} \alpha_j}y_0^{\beta_0\beta_1\ldots\beta_{j-1}},
\eeq 
and consequently,
\begin{align}
\label{AS3}
\limsup_{j\to\infty} y_j &\le A^{ B\bar\alpha}y_0^{\bar \beta},
\end{align}
where $B_{j}=\max\{1,\beta_m\beta_{m+1}\beta_{m+2}\ldots\beta_{n}:1\le m\le n< j\}$, and
$ B=\limsup_{j\to\infty} B_j$.
\end{lemma}
%==================================================================%
\begin{proof}
Applying \eqref{A2} recursively, we have
\begin{align*}
 y_{j+1}
&\le A^{\alpha_j} (A^{\alpha_{j-1}} y_{j-1}^{\beta_{j-1}})^{\beta_j}=A^{\alpha_j+\alpha_{j-1}\beta_j}y_{j-1}^{\beta_{j-1} \beta_j}\\
&\le \dots\le A^{\alpha_j+\alpha_{j-1}\beta_j + \alpha_{j-2}\beta_{j-1}\beta_j +\alpha_{j-3}\beta_{j-2}\beta_{j-1}\beta_j +\ldots+\alpha_0\beta_1\beta_2\ldots\beta_j }
\cdot y_0^{\beta_0\beta_1\ldots\beta_j}.
\end{align*}
It follows that $y_{j+1}\le A^{B_{j+1}(\sum_{i=0}^j  \alpha_i)}y_0^{\beta_0\beta_1\ldots\beta_j}.$
Hence, we obtain \eqref{AS}. Taking the limit superior of \eqref{AS} as $j \to \infty$, we obtain \eqref{AS3}.
Note that $B<\infty$ by Cauchy's criterion.
\end{proof}
%==================================================================%

The next lemma is the main adaptation used in this paper.

\begin{lemma}\label{Genn}
Let $\kappa_j>0$, $s_j \ge r_j>0$ and $\omega_j\ge 1$  for all $j\ge 0$.
Suppose there is $A\ge 1$ such that
\beqs
y_{j+1}\le A^\frac{\omega_j}{\kappa_j} (y_j^{r_j}+y_j^{s_j})^{\frac 1{\kappa_j}}\quad\forall j\ge 0.
\eeqs
Denote $\beta_j=r_j/\kappa_j$ and $\gamma_j=s_j/\kappa_j$.
Assume
\beqs
\bar\alpha \eqdef \sum_{j=0}^{\infty}  \frac{\omega_j}{\kappa_j}<\infty\text{ and the products }
\prod_{j=0}^\infty \beta_j, 
\prod_{j=0}^\infty \gamma_j\text{ converge to positive numbers } \bar\beta,\bar \gamma, \text{ resp.}
\eeqs
Then
\beq\label{doub1}
y_j\le (2A)^{G_{j} \bar \alpha} \max\Big\{  y_0^{\gamma_0\ldots\gamma_{j-1}},y_0^{\beta_0\ldots\beta_{j-1}} \Big\}\quad \forall j\ge 1,
\eeq
where $G_j=\max\{1, \gamma_{m}\gamma_{m+1}\ldots\gamma_{n}:1\le m\le n< j\}$.
Consequently,
\beq\label{doub2}
\limsup_{j\to\infty} y_j\le (2A)^{G \bar\alpha} \max\{y_0^{\bar \gamma},y_0^{\bar \beta}\},
\quad\text{where  }G=\limsup_{j\to\infty} G_j.
\eeq
%=\sup \{1, \gamma_{m}\gamma_{m+1}\ldots\gamma_{n}:n\ge 1 \text{ and }1\le m\le n\}$.
\end{lemma}

%====================================================================%

\begin{proof}
We prove \eqref{doub1} first. Define a sequence $(z_j)_{j=0}^\infty$ by $z_0=y_0$ and $z_{j+1}= A^{\frac {\omega_j}{\kappa_j}}(z_j^{r_j}+z_j^{s_j})^{\frac 1{\kappa_j}}$ for $j\ge 0$.
Then 
\beq\label{yz} 
y_j\le z_j \text{ for all }j\ge 0.
\eeq
Therefore, it suffices to bound $z_j$. We consider three cases.

\underline{Case 1}: $z_0\ge 1$. Clearly, $z_j\ge 1$ for all $j$. Hence, 
\begin{align*}
z_{j+1}\le A^{\frac {\omega_j}{\kappa_j}} (2z_j^{s_j})^\frac1{\kappa_j}\le (2A)^{\frac {\omega_j}{\kappa_j}}z_j^{\gamma_j}.
\end{align*}
Then using Lemma \ref{lemb1} we have
\beq\label{caseone}
z_j \le (2A)^{G_j \sum_{i=0}^{j-1} \omega_i/\kappa_i} z_0^{\gamma_0\gamma_1\ldots\gamma_{j-1}}.
\eeq
Together with \eqref{yz}, we obtain \eqref{doub1}.

\underline{Case 2}: $z_j<1$ for all $j\ge 0$. Then 
\begin{align*}
z_{j+1}\le A^{\frac {\omega_j}{\kappa_j}} (2z_j^{r_j})^\frac1{\kappa_j}\le (2A)^{\frac {\omega_j}{\kappa_j}}z_j^{\beta_j}.
\end{align*}
Applying Lemma \ref{lemb1} gives
\beq\label{casetwoa}
z_j\le (2A)^{B_j\sum_{i=0}^{j-1} \omega_i/\kappa_i} z_0^{\beta_0\beta_1\ldots\beta_{j-1}},
\eeq
where $B_j=\max\{1, \beta_m\beta_{m+1}\ldots\beta_{n}:1\le m\le n< j\}<\infty.$
Note that $B_j\le G_j$. Then, again, \eqref{casetwoa} and \eqref{yz} yield \eqref{doub1}.

\underline{Case 3}: There is $j_0\ge 1$ such that $z_j<1$ for $0\le j<j_0$ and $z_{j_0}\ge 1$. 
Applying \eqref{casetwoa} to $1\le j\le j_0$,
\beq\label{case3small}
z_j\le (2A)^{B_j\sum_{i=0}^{j-1} \omega_i/\kappa_i} z_0^{\beta_0\beta_1\ldots\beta_{j-1}}
\le (2A)^{G_j \sum_{i=0}^{j-1} \omega_i/\kappa_i}  z_0^{\beta_0\beta_1\ldots\beta_{j-1}}.
\eeq

Same as Case 1, $z_j\ge 1$ for all $j\ge j_0$. Then applying \eqref{caseone} for $j> j_0$ gives
\beq\label{case2a}
z_j\le (2A)^{G_{j_0,j} \sum_{i=j_0}^{j-1}  \omega_i/\kappa_i} z_{j_0}^{\gamma_{j_0}\gamma_{j_0+1}\ldots\gamma_{j-1}},
\eeq
where $G_{j_0,j}=\sup\{1, \gamma_{j_0+m}\gamma_{j_0+m+1}\ldots\gamma_{j_0+n}: 1\le m\le n< j-j_0\}<\infty.$
Using  inequality \eqref{case3small} with $j=j_0$ to estimate $z_{j_0}$ in \eqref{case2a}, we have for $j>j_0$ that
\begin{align*}
z_j
&\le (2A)^{G_{j_0,j} \sum_{i=j_0}^{j-1}  \omega_i/\kappa_i} \Big\{ (2A)^{G_{j_0} \sum_{i=0}^{j_0-1} \omega_i/\kappa_i} z_{0}^{\beta_0\beta_1\ldots \beta_{j_0-1}} \Big\}^{\gamma_{j_0}\gamma_{j_0+1}\ldots\gamma_{j-1}}\\
&= (2A)^{G_{j_0,j} (\sum_{i=j_0}^{j-1}  \omega_i/\kappa_i) + G_{j_0} \gamma_{j_0}\gamma_{j_0+1}\ldots\gamma_{j-1} (\sum_{i=0}^{j_0-1} \omega_i/\kappa_i)}  z_{0}^{\beta_0\beta_1\ldots \beta_{j_0-1}\gamma_{j_0}\gamma_{j_0+1}\ldots\gamma_{j-1}}.
\end{align*}
Since $z_0<1$, $\beta_i\le \gamma_i$ for all $i$, and $G_{j_0,j}, G_{j_0} \gamma_{j_0}\gamma_{j_0+1}\ldots\gamma_{j-1}  \le  G_{j}$, we obtain
\beq\label{casetwob}
\begin{aligned}
z_j
&\le (2A)^{\max\{G_{j_0,j} ,G_{j_0} \gamma_{j_0}\gamma_{j_0+1}\ldots\gamma_{j-1} \} \sum_{i=0}^{j-1} \omega_i/\kappa_i}  z_{0}^{\beta_0\beta_1\ldots \beta_{j-1}}
\le (2A)^{G_{j}  \sum_{i=0}^{j-1}\omega_i/\kappa_i} z_{0}^{\beta_0\beta_1\ldots \beta_{j-1}}
\end{aligned}
\eeq
for all $j>j_0$. Then, \eqref{case3small}, \eqref{casetwob} and \eqref{yz} imply \eqref{doub1}. This completes the proof of \eqref{doub1} for all cases.

Now, taking the limit superior of \eqref{doub1} as $j\to \infty$ yields \eqref{doub2}. 
\end{proof}
%================================================================%

Note that  the  numbers $G,\bar\alpha,\bar\beta$ and $\bar \gamma$ in inequality \eqref{doub2} are explicitly defined.

\medskip
\textbf{Acknowledgment.} L. H. acknowledges the support by NSF grant DMS-1412796.

%%%%%%%%%%%%%%%%%%%%%%%%%%%%%%%%%%%%%%%%%%%%%%%%%%%%%%%%%%%%%%%%%%%%%
% Bibliography using BibTeX
%%%%%%%%%%%%%%%%%%%%%%%%%%%%%%%%%%%%%%%%%%%%%%%%%%%%%%%%%%%%%%%%%%%%%
%%%%%%%%%%%%%%%%%%%%%%%%%%%%%%%%%%%%%%%%%%%%%%%%%%%%%%%%%%%%%%%%%%%%%
\myclearpage

%\bibliography{luanworks,paperbase}{}
\bibliographystyle{abbrv}

%%%%%%%%%%%
%%%%%%%%%%%
\def\cprime{$'$}

%%%%%%%%%%%
%%%%%%%%%%%

\end{document}